\documentclass[11pt]{amsart}
\usepackage[
  letterpaper,
  left=1.7in,
  right=1.7in,
  top=1.1in,
  bottom=1.1in
]{geometry}

\title{Three-manifolds that do not admit veering triangulations}
\author{\textsc{Layne Hall}
}
\email{layne.hall@warwick.ac.uk}

\usepackage{amssymb}
\usepackage{amsmath}
\usepackage{amsthm}
\usepackage{mathtools}
\usepackage[font={small}]{caption}
\usepackage{subcaption}
\usepackage{algorithmicx}
\usepackage[noend]{algpseudocode}
\usepackage{algorithm}
\usepackage{adjustbox} 
\usepackage{booktabs} 
\usepackage{enumerate}

\usepackage[hidelinks]{hyperref}
\RequirePackage{xcolor}
\definecolor{link}{RGB}{152, 90, 191}
\definecolor{cite}{RGB}{73, 184, 110} 
\definecolor{url}{RGB}{230, 163, 122} 
\usepackage[hidelinks]{hyperref}
\hypersetup{
    colorlinks=true,
    linkcolor=link,
    citecolor=cite,
    urlcolor=url
    }
\usepackage{enumerate}
\usepackage{enumitem}
\setlist[enumerate,1]{label=\arabic*.,ref=\arabic*}
\usepackage{textcomp}
\usepackage{titlesec}
\usepackage{nameref}
\usepackage{lipsum}
\usepackage[pagewise]{lineno} 
\usepackage{xcolor, soul} 
\sethlcolor{green} 

\usepackage[final]{microtype} 
\usepackage[capitalise]{cleveref}

\theoremstyle{plain}

\usepackage{thmtools}
\newtheorem{thm}{Theorem}[section]
\newtheorem{prop}[thm]{Proposition}
\newtheorem{lemma}[thm]{Lemma}

\newtheorem{claim}[thm]{Claim}

\newtheorem{fact}[thm]{Fact}

\makeatletter
\let\c@algorithm\relax 
\makeatother

\usepackage{aliascnt}
\newaliascnt{algorithm}{thm} 

\theoremstyle{definition}
\newtheorem{algo/}[thm]{Algorithm}
\newtheorem{con/}[thm]{Construction}
\newtheorem{conj/}[thm]{Conjecture}
\newtheorem{question/}[thm]{Question}
\newtheorem{example/}[thm]{Example}
\newtheorem{rem/}[thm]{Remark}
\newtheorem{defn/}[thm]{Definition}
\newtheorem{observation/}[thm]{Remark}
\newtheorem{notation/}[thm]{Notation}
\newtheorem{convention/}[thm]{Convention}
\newtheorem{move/}[thm]{Move}
\newtheorem{process/}[thm]{Process}
\newtheorem{procedure/}[thm]{Procedure}
\newtheorem{sketch/}[thm]{Sketch}
\newtheorem{algothm/}[thm]{Algorithm}
\newtheorem{obs/}[thm]{Observation}

\usepackage{graphicx}

\newenvironment{defn}
  {%
   \pushQED{\qed}\begin{defn/}}
  {\popQED\end{defn/}}

\newenvironment{example}
  {%
   \pushQED{\qed}\begin{example/}}
  {\popQED\end{example/}}

\newenvironment{rem}
  {%
   \pushQED{\qed}\begin{rem/}}
  {\popQED\end{rem/}}

\newenvironment{con}
  {%
   \pushQED{\qed}\begin{con/}}
  {\popQED\end{con/}}

\newcommand{\cal}[1]{\mathcal{#1}}
\newcommand{\lift}[1]{\widetilde{#1}}

\newcommand{\cusps}{\Delta}

\DeclareMathOperator{\windingorder}{<_{\gw}}

\DeclareMathOperator{\bbZ}{\mathbb{Z}}
\DeclareMathOperator{\gw}{\Theta}

\numberwithin{algorithm}{section}
\numberwithin{equation}{section}

\titleformat{\title}
{\normalfont\scshape\centering}{\title}{1em}{}
\titleformat{\section}
{\normalfont\scshape\centering}{\thesection}{1em}{}
\titleformat{\subsection}
{\normalfont\bfseries}{\thesubsection}{1 em}{}

\usepackage{graphicx}
\graphicspath{ {./figures/} }

\usepackage[
backend=biber,
style=alphabetic,
citestyle=alphabetic,
maxbibnames=99
]{biblatex}
\DeclareMultiCiteCommand{\parencites}[\mkbibbrackets]
  {\parencite}{\addsemicolon\space}
  
\theoremstyle{plain}
\newtheorem*{thm:censusnoveering}{\Cref{thm:census_no_veering}}
\newtheorem*{thm:smoothcriterion}{\Cref{thm:smoothcriterion}}
\newtheorem*{thm:infiniteegs}{\Cref{thm:infiniteegs}}
\newtheorem*{thm:totalorder}{\Cref{thm:totalorder}}
\newtheorem*{thm:biordercensus}{\Cref{thm:biordercensus}}

\begin{document}
\begin{abstract}
    We show that if a three-manifold $M$ admits a veering triangulation, then the dual spine of every essential triangulation of $M$ can be smoothed to a branched surface. We use this criterion to identify a large collection of cusped hyperbolic three-manifolds that do not admit veering triangulations. We then use the same criterion to construct the first infinite family of cusped hyperbolic three-manifolds without veering triangulations. 
\end{abstract}

\maketitle

\markboth{\textsc{Layne Hall}}{\textsc{Obstructing with branchings}}

\section{Introduction}

A veering triangulation of a three-manifold is a rigid combinatorial object that closely reflects the structure of the manifold \cite{fg-angle-veering, ss-link, landry-minsky-taylor, parlak-taut-alex, agol-tsang}. A natural problem is to understand which manifolds admit such triangulations. Large-scale surveys suggest that cusped hyperbolic three-manifolds that admit a veering triangulation are rare \cite{veering-census, shana_10_tets}. However, only a handful of these manifolds are known not to admit a veering triangulation \cite{misha-dunfield-egs,ss-veering-finiteness}. We give the first local criterion, \emph{branchability}, that obstructs the presence of veering triangulations.

\begin{thm:smoothcriterion}
    Suppose that $M$ admits a veering triangulation. Let $\cal{T}$ be an essential ideal triangulation of $M$. Then $\cal{T}$ is branchable. 
\end{thm:smoothcriterion}

Thus, if $M$ admits a single essential triangulation that is not branchable, then $M$ does not admit a veering triangulation. The property of being branchable is a simple combinatorial condition: the dual spine of the triangulation can be smoothed to give a branched surface in $M$.

Schmalian gave the first few examples of three-manifolds that do not admit veering triangulations \cite{misha-dunfield-egs}. Given an oriented cusped hyperbolic three-manifold $M$, they consider every double cover of $M$. They then compare fillings of these covers to Dunfield's computed data for closed manifolds that do not admit cooriented taut foliations \cite{dunfield-foliar}. In particular, the test requires a large amount of data on closed manifolds.

Schleimer-Segerman recently gave a means of counting veering triangulations on certain manifolds \cite{ss-veering-finiteness}. However, it also only rules out veering triangulations for a small set of examples \cite[Proposition 9.7]{ss-veering-finiteness}.

In contrast, our criterion is readily applicable to large datasets. Shana Li \cite{shana_11_tets} recently compiled a census of all orientable cusped hyperbolic three-manifolds up to eleven tetrahedra. Applying our criterion, we find:

\begin{thm:censusnoveering}
    At least 165,065 manifolds in the census of orientable cusped hyperbolic three-manifolds up to eleven tetrahedra do not admit veering triangulations.
\end{thm:censusnoveering}
These examples constitute roughly 23\% of the 717,993 manifolds in the census.

There are also simple local obstructions to branchability. Such obstructions can persist under certain `cut-and-paste' operations. We use this to build the first explicit infinite family of hyperbolic three-manifolds that do not admit veering triangulations.
\begin{thm:infiniteegs}
     There are infinitely many cusped hyperbolic three-manifolds that do not admit veering triangulations. 
\end{thm:infiniteegs}

Our strategy for proving \Cref{thm:smoothcriterion} is to relate branchings to invariant total orders on the cusps $\cusps(\lift{M})$ of the universal cover. Showing that such an order yields a branching is straightforward (\Cref{prop:cusps_give_branching}). Instead, the main technical step in proving \Cref{thm:smoothcriterion} is finding an invariant total order. A veering triangulation has two natural branchings given by the \emph{veering branched surfaces}. When these branched surfaces are cooriented, each imparts a partial order on the cusps of $\lift{M}$. We show that this can be completed to a total order.

\begin{thm:totalorder}
    Suppose that $M$ admits an edge-orientable veering triangulation $\cal{V}$. Then there is a $\pi_1(M)$-invariant total order on $\cusps(\lift{M})$.
\end{thm:totalorder}

The fact that this order extends the partial order is \Cref{prop:compatible_order}. The relevant definitions are given precisely in \Cref{sec:branchings}. We note that \Cref{thm:totalorder} requires coorientability of the branched surface, but that \Cref{thm:smoothcriterion} does not. We bypass coorientability by choosing a specific total order that interacts nicely with double covers (\Cref{lem:twisted_order_flip}). 

The existence of an invariant total order on the cusps of $\lift{M}$ can be formulated as an orderability property of $\pi_1(M)$ (\Cref{rem:convexity}). Thus, the obstruction is in some sense analogous to the closed case, where pseudo-Anosov flows are obstructed by the absence of circular orders \cite{calegari-dunfield}.

To prove \Cref{thm:totalorder}, we study the \emph{loom space} of a veering triangulation \cite{loom_new_arxiv}. This is a \emph{bifoliated plane} that can equivalently be viewed as the universal cover of the drilled orbit space of a pseudo-Anosov flow without perfect fits. In this light, \Cref{thm:totalorder} is a cusped analogue of recent work that uses foliations and bifoliated planes to build left orders on the fundamental groups of closed three-manifolds \cite{zhao-orderability, camargo-lu-lobfp}. Our proof is not a direct adaptation of these results, but it borrows the core idea: count signs on piecewise paths between the cusps.

Invariant cusp orders also give a connection between branchings and \emph{bi-orders}. In general, bi-orderability of three-manifold groups is less understood than left-orderability, with most work focussing on knot and link complements \cite{clay-rolfsen, kin-rolfsen}. We show that a bi-order on $\pi_1(M)$ implies the existence of cooriented branchings (\Cref{prop:bi-order_branching}). We use this to obstruct bi-orderability in the same way we obstruct \emph{edge-orientable} veering triangulations. As an application, we provide large-scale data on hyperbolic census manifolds whose fundamental groups are not bi-orderable.

\begin{thm:biordercensus}
    At least 570,207 manifolds in the census of orientable cusped hyperbolic three-manifolds up to eleven tetrahedra do not have bi-orderable fundamental groups. These manifolds also do not admit edge-orientable veering triangulations.
\end{thm:biordercensus}

\textbf{Outline.}
We describe triangulations, branchings and cusp orders in \Cref{sec:branchings}. The connection to bi-orders is made in \Cref{subsec:orders}. In \Cref{sec:medians}, we recall the combinatorial geometry of loom spaces. Then we use the geometry to show that loom spaces have medians. This proof requires a careful examination of several cases. 

In \Cref{sec:totalorder}, we build a total order to prove \Cref{thm:totalorder}. The order defined in \Cref{def:winding_order} requires only the basics of loom space geometry. The majority of the work is in showing that the order is transitive. To establish this, we use the existence of medians to prove the cycle identity in \Cref{lem:cycle_identity}. Transitivity is then a straightforward consequence of the cycle identity. We then use the total order to deduce \Cref{thm:smoothcriterion}. 

In \Cref{sec:examples}, we describe the numerical census examples. We then build an explicit infinite family to prove \Cref{thm:infiniteegs}. The construction uses a natural Dehn surgery operation on taut triangulations (\Cref{con:mobius_surgery}). 

\textbf{Acknowledgements.} The author thanks Saul Schleimer and Henry Segerman for many discussions and clarifications about the structure of loom spaces. We also thank them for sharing an early version of their work on counting veering triangulations and for comments on an earlier version of this manuscript. 

This work took place while the author was supported by the Warwick Mathematics Institute Centre for Doctoral Training. The author gratefully acknowledges funding from the University of Warwick and the UK Engineering and Physical Sciences Research Council (Grant number: EP/W524645/1).

\textbf{AI declaration.} In the final stages of writing, the author used ChatGPT 5.6 Sol to point out possible copy-editing errors and clarity concerns in the exposition. The same model was used to suggest improvements to the code that we used to generate computational examples within the paper. 

\section{Triangulations and  branchings}\label{sec:branchings}
We define veering triangulations and branchings. Then we relate branchings to orientations on edges, total cusp orders and bi-orders.

\subsection{Veering and essential triangulations}
Let $M$ be an irreducible oriented compact three-manifold with boundary a disjoint union of tori. Veering triangulations were introduced by Agol \cite[Definition 4.1]{agol-veering}, and generalised in \cite[Definition 1.3]{hodgson-veering}. Recall that an \emph{ideal tetrahedron} is an oriented tetrahedron minus its vertices. An \emph{ideal triangulation} of $M$ is a collection of ideal tetrahedra with face pairings such that the realisation is homeomorphic to the interior of $M$ \cite{thurston-notes}.

    A \emph{veering tetrahedron} is an ideal tetrahedron $t$ equipped with:
    \begin{itemize}
        \item A transverse taut structure: an assignment of dihedral angles in $\{\pi,0\}$ such that the sum of the interior angles of every truncated vertex is $\pi$. Moreover, there is a coorientation of the faces of $t$, with two pointing into $t$ and two pointing out of $t$, such that faces sharing a $\pi$-edge have the same coorientation. 
    
        \item A veering colouring: a colouring of the edges in red and blue that satisfies the following condition. For a vertex $v$ of $t$, let $e$ be the edge incident to $v$ with dihedral angle $\pi$. Using the orientation of $t$, view $\partial t$ from outside $t$. Suppose that $e, e', e''$ are the edges that meet $v$ in counter-clockwise order. Then $e'$ is coloured blue and $e''$ is coloured red.
    \end{itemize}
    This does not specify the colours of the upper and lower edges of a veering tetrahedron. So, there are four combinatorially distinct veering tetrahedra. 
    \begin{figure}[h]
        \centering
        \includegraphics[width=0.7\linewidth]{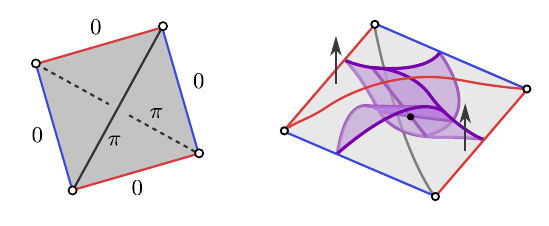}
        \caption{Left: a veering tetrahedron. Right: the upper veering branched surface of a veering tetrahedron with a red top edge.}
        \label{fig:veering_tet}
    \end{figure}
    
    \begin{defn}\label{def:veering}
    A \emph{veering triangulation} is an ideal triangulation by veering tetrahedra such that:
    \begin{enumerate}
        \item The sum of the dihedral angles around each edge in the triangulation is $2\pi$.
        \item The face pairings identify faces cooriented into their tetrahedra with those cooriented out of their tetrahedra.
        \item The face pairings preserve the edge-colouring. \qedhere
    \end{enumerate}
    \end{defn}

A triangulation consisting of transverse taut tetrahedra satisfying the first two items of \Cref{def:veering} is a \emph{transverse taut triangulations} \cite{lackenby-taut}. Further relaxing the condition of face coorientations defines a \emph{taut triangulation}.
\begin{rem}
    If $M$ admits a veering triangulation, then $M$ is a cusped hyperbolic three-manifold by \cite[Theorem 1.5]{hodgson-veering}.
\end{rem}
\begin{rem}
Veering triangulations are in correspondence with \emph{pseudo-Anosov flows without perfect fits} on closed manifolds. The flows are obtained by Dehn filling the cusped manifold along sufficiently far slopes. Conversely, the triangulation of the cusped manifold is realised by drilling the singular orbits of the flow. We refer to \cite{barthelme-mann-book} for background on pseudo-Anosov flows, and \cite{ss-link} for a discussion of this correspondence.
\end{rem}

Let $\cal{T}$ be an ideal triangulation of $M$. A topological feature of certain ideal triangulations is essentiality. Let $\dot{\cal{T}}$ be the truncated triangulation of $M$ obtained by truncating the vertices of $\cal{T}$. For an edge $e$, let $\dot{e}$ be the truncated image of $e$ as an edge in $\dot{\cal{T}}$.
\begin{defn}
     We say that an edge $e$ of $\cal{T}$ is \emph{inessential} if $\dot{e}$ is homotopic, relative to the boundary of $M$, into the boundary of $M$. 

    If no edges of $\cal{T}$ are inessential, we say that $\cal{T}$ is \emph{essential}. 
\end{defn}
  For a broader discussion, see \cite[Definitions 3-5]{edge_essential}.

By \cite[Theorem 16]{edge_essential}, all taut triangulations, and hence veering triangulations, are essential. The same theorem implies that geometric triangulations are also essential. In contrast, a triangulation with a degree-one edge $e$ has an inessential edge that `links' $e$ \cite[\S 3]{henry_deg_one_edges}. 

Essential triangulations interact nicely with the boundary components on the universal cover.
\begin{defn}\label{def:cusps}
Let $\lift{M}$ be the universal cover of $M$. Call the connected boundary components of $\lift{M}$ the \emph{cusps} of $\lift{M}$, denoted $\cusps(\lift{M})$.
\end{defn}

 If $\cal{T}$ is essential, the endpoints of the edges of $\lift{\cal{T}}$ are distinct cusps in $\cusps(\lift{M})$. When $M$ is hyperbolic, this definition of a cusp corresponds to the identification of a cusp of $\lift{M}$ with a parabolic fixed point in $\partial \mathbb{H}^3$.

\subsection{Branchings}
  Following \cite{gabai-oertel}, a \emph{branched surface} in $M$ is an embedded compact two-complex locally modelled on the neighbourhoods shown in \Cref{fig:local_models}. The vertices of a branched surface are \emph{triple points}. 

  \begin{figure}[h]
    \centering
    \includegraphics[width=0.8\linewidth]{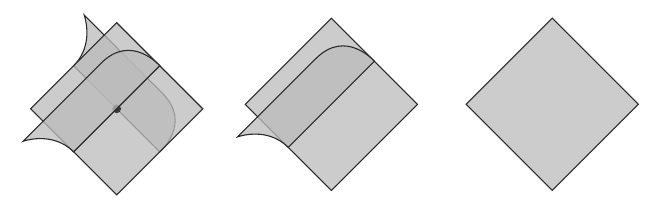}
    \caption{The local models of a branched surface.}
    \label{fig:local_models}
\end{figure}

    Let $D(\cal{T})$ be the dual complex of $\cal{T}$. Recall that $D(\cal{T})$ has the structure of a \emph{special spine} \cite[Theorem 1.1.26]{matveev}. The dual spine of a tetrahedron can be smoothed to the local neighbourhood of a triple point in a branched surface. See \Cref{fig:foliar_tet}, left. A tetrahedron has twelve distinct such smoothings. 
    
\begin{defn}
A \emph{branching} on $\cal{T}$ is a choice of smoothing of each vertex of $D(\cal{T})$ that agrees across faces, giving $D(\cal{T})$ the structure of a branched surface. If $\cal{T}$ admits a branching, we say that $\cal{T}$ is \emph{branchable}.

If $\cal{T}$ admits a branching for which the branched surface is coorientable, we say that $\cal{T}$ is \emph{coorientably branchable}. 
\end{defn}

By a slight abuse of terminology, we also call the branched surface itself a branching. Note that we do not impose global constraints on the branched surface, such as carrying a lamination.

A fundamental example is the following.

\begin{example}\label{eg:vbs}
    A veering triangulation has two natural branchings given by the \emph{veering branched surfaces} \parencites[\S 6.1]{ss-link}[\S 4.2]{lmt-polynomial}.  Call these the \emph{stable} and \emph{unstable} branchings of $\cal{V}$. Denote them by $B^s$ and $B^u$. 
    
    We describe the train track of the unstable branching on the boundary of each tetrahedron with respect to the veering structure. On the lower two faces, the track points to the lower $\pi$-edge. On the upper two faces, it points to the equatorial edges of the same colour as the upper $\pi$-edge. See \Cref{fig:veering_tet}. The stable branched surface is defined with the roles of `upper' and `lower' interchanged. 
    
    When one of the branchings is coorientable, both are, and $\cal{V}$ is called \emph{edge-orientable}. The veering branched surfaces are \emph{dynamic} in the sense of Christy and Mosher. Moreover, they carry the blow-ups of the weak-stable and weak-unstable foliations of the filled pseudo-Anosov flow \cite{christy, mosher1996laminations,  ss-pair, agol-tsang}.
\end{example}

\subsection{Edge orientations}
We encode cooriented branchings with edge orientations, following  Dunfield's persistently foliar triangulations \cite[\S 8]{dunfield-foliar}.

    Let $B$ be a branching on $\cal{T}$. When $B$ is cooriented, the coorientation of each sector induces an orientation on the sector's dual edge. These orientations are preserved under face pairings. We call this the \emph{edge orientation} of $B$. See \Cref{fig:foliar_tet}.

    If $B$ is a non-coorientable branching, there is a double cover of the underlying manifold under which the lift of $B$ is coorientable (see, for example, \cite[Lemma 5.4]{fg-angle-veering}). For $B^s$ and $B^u$ the veering branched surfaces, this is the same cover, called the \emph{edge-orientable double cover} \cite[\S 4]{parlak-taut-alex}.

    \begin{figure}[h]
        \centering
        \includegraphics[width=0.6\linewidth]{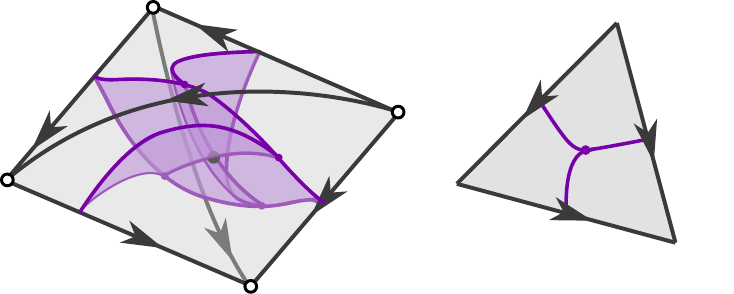}
        \caption[The model of a cooriented branched tetrahedron]{Left: A cooriented branching structure on a tetrahedron. The edge orientations determine the coorientation; ignoring them gives an uncooriented branching. Right: A single face, showing the train track pointing from two short edges to one long edge.}
        \label{fig:foliar_tet}
    \end{figure}
    
\begin{defn}
    Let $B$ be a cooriented branching on $\cal{T}$. The one-skeleton of $\lift{\cal{T}}$ with the lifted edge orientations induced by $B$ has the structure of a directed graph. We say that $B$ is \emph{acyclic} if this directed graph is acyclic.

    Suppose that $B$ is acyclic. For distinct $c,d \in \cusps(\lift{M})$ set $c<_Bd$ if there is a directed path from $c$ to $d$. Call this partial order the \emph{edge order} of $B$. 
\end{defn}

Note that if $B$ is acyclic, the triangulation $\cal{T}$ is necessarily essential. 

\begin{lemma}\label{lem:lamination_implies_acyclic}
    Suppose that $B$ fully carries a cooriented taut lamination. Then $B$ is acyclic.
\end{lemma}
\begin{proof}
    Suppose for contradiction that $B$ is cyclic. Let $\Lambda$ be the carried lamination. Lift $\Lambda$ to $\lift{\Lambda}$ on $\lift{M}$. We build a loop $\alpha$ by connecting the edges in the cycle with paths that run close to the boundary. Since $\lift{\Lambda}$ is fully carried, $\alpha$ is positively transverse to $\lift{\Lambda}$. Then $\alpha$ descends to a null-homotopic loop in $M$ that is positively transverse to $\Lambda$, contradicting \cite[Theorem 1 (d)]{gabai-oertel}. 
\end{proof}

The veering branched surfaces of \Cref{eg:vbs} carry laminations. 

\begin{lemma}\label{lem:veer_acyclic}
    The stable and unstable veering branched surfaces of an edge-orientable veering triangulation are acyclic. \qed
\end{lemma}

We give a cyclic example. 

\begin{example}\label{eg:m006}
    Let $M$ be the census manifold $\texttt{m006}$. Its census triangulation is not branchable. A single Pachner $2\to3$ move yields a triangulation with a coorientable branching. However, the branching is cyclic. See \Cref{fig:local_obs}. 
\begin{figure}[h]
    \centering
    \includegraphics[width=0.7\linewidth]{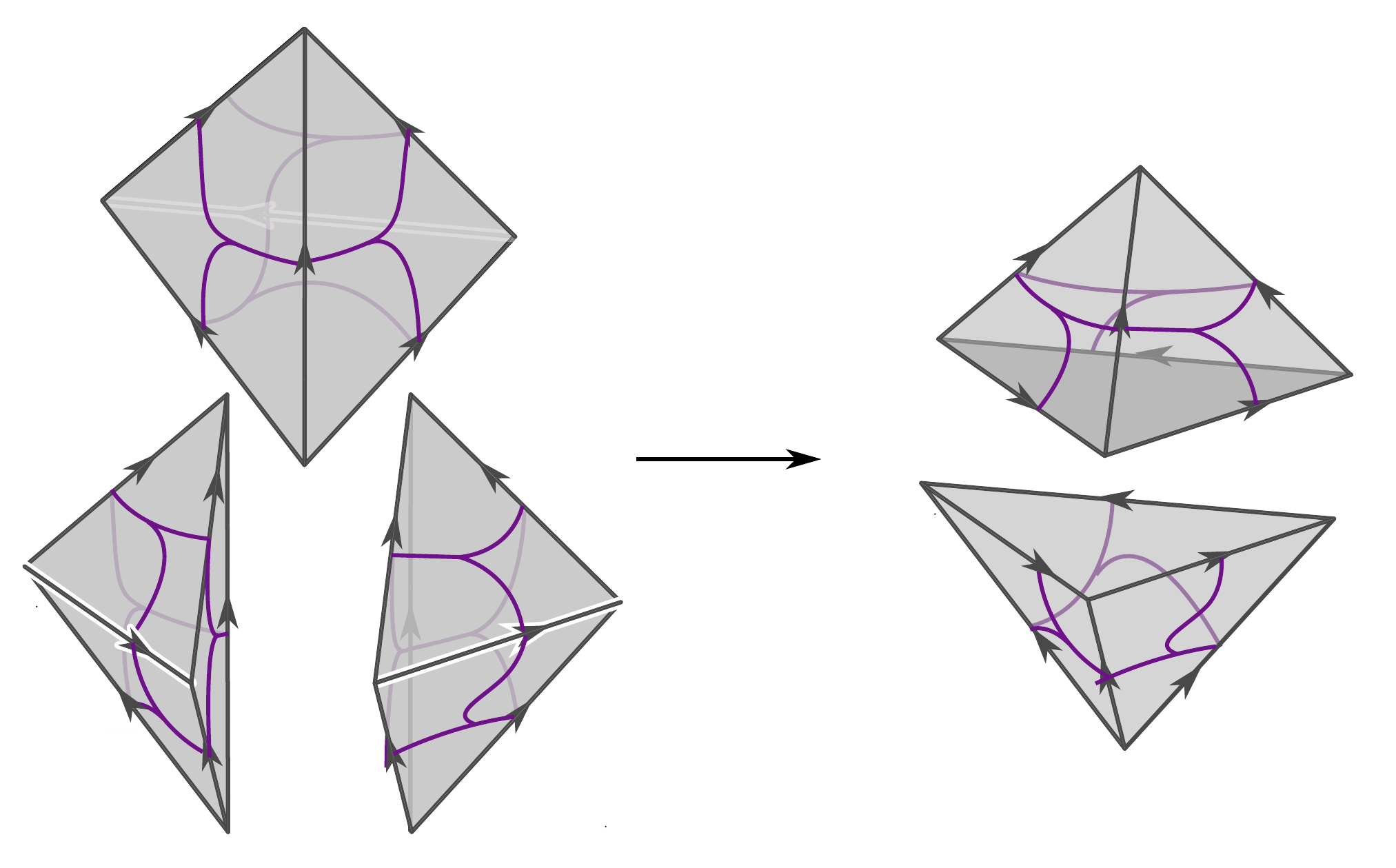}
    \caption[A cyclic branching]{Left: A cyclic cooriented branching. Right: Under a Pachner $3\to2$ move, there is no induced branching.}
    \label{fig:local_obs}
\end{figure}
\end{example}


\begin{defn}
    Let $B$ be an acyclic cooriented branching on $\cal{T}$. A total order on $\cusps(\lift{M})$ is \emph{compatible} with $B$ if it extends $<_{B}$. 
\end{defn}

It is not clear in general whether an acyclic branching admits a compatible order that is also $\pi_1(M)$-invariant. \Cref{thm:totalorder} asserts that we can extend the partial order when it arises from a veering branched surface.

Conversely, edge orientations give cooriented branchings as follows.
\begin{con}\label{con:edge_o_to_branch}
    Suppose that $\cal{T}$ admits an orientation of its edges such that the restriction to each tetrahedron of $\cal{T}$ is acyclic. We call the orientation \emph{locally acyclic}.
    
    In a locally acyclic orientation, each face $f$ has a maximal and a minimal vertex with respect to the edge orientations. There is a unique edge of $f$ from the minimal to the maximal vertex; call this edge \emph{long}. The other two edges are \emph{short}. Place a train track in $f$ that points to the long edge. In a tetrahedron $t$, the train tracks on the faces interpolate uniquely to a branching on the dual spine of $t$. See  \Cref{fig:foliar_tet}. The branching is cooriented by the edge orientations.
\end{con}

Summarising the discussion, we have:
\begin{fact}\label{lem:order_gives_branch}
    There is a bijection between locally acyclic edge orientations on $\cal{T}$ and cooriented branchings on $\cal{T}$. \qed 
\end{fact}

We record the following basic fact that we will use to descend cooriented branchings to non-coorientable ones.

\begin{lemma}\label{lem:smooth_flipped}
The two branchings on a tetrahedron given by an edge orientation and its reverse induce the same branching with opposite coorientations.
\end{lemma}

\begin{proof}
    The long edge in a face is the same after reversing the arrows. Then the patterns of the two branchings agree on the boundary of every tetrahedron. Thus, the branchings agree, and the reversed arrows reverse the coorientation. 
\end{proof}

\subsection{Orders}\label{subsec:orders}

Branchings arise naturally from orders on the cusps of $\lift{M}$.  

\begin{prop}\label{prop:cusps_give_branching}
    Suppose that $\cusps(\lift{M})$ admits a $\pi_1(M)$-invariant total order. Then any essential ideal triangulation of $M$ is coorientably branchable.
\end{prop}
\begin{proof}
    Let $\cal{T}$ be an essential ideal triangulation of $M$. Since $\cal{T}$ is essential, a lifted edge $\lift{e}$ of $\lift{\cal{T}}$ has endpoints on distinct cusps. So, we can orient $\lift{e}$ to point to the greater cusp in the invariant total order. This orientation is $\pi_1(M)$-invariant, so it descends to an orientation of $\cal{T}$ that is locally acyclic. Now apply \Cref{lem:order_gives_branch}. 
\end{proof}

\Cref{prop:cusps_give_branching} may hold even for inessential triangulations. In numerical searches, we have been unable to find an unbranchable triangulation of a manifold that we know admits a $\pi_1(M)$-invariant total order on $\cusps(\lift{M})$.

\begin{rem}\label{rem:pachner_orders}
    Suppose that $\cal{T}$ and $\cal{T}'$ are essential triangulations of the same manifold. Suppose there is an invariant total order on $\cusps(\lift{M})$. Then $\cal{T}$ and $\cal{T}'$ inherit branchings $B$ and $B'$, and we can connect these branchings by a sequence of Pachner moves. Namely, Kalelkar-Schleimer-Segerman show that $\cal{T}$ and $\cal{T}'$ can be connected by a sequence of Pachner moves such that every intermediate triangulation along the path is also essential \cite[Theorem 1.3]{essential_moves_II}. Each such triangulation inherits a branching from the total order on $\cusps(\lift{M})$, and this branching agrees on all tetrahedra that were left unchanged by the move. Compare with \Cref{eg:m006}.  
\end{rem}

For the remainder of the subsection, we take a brief detour to relate cusp orders to orderability properties on $\pi_1(M)$. Recall that $\pi_1(M)$ is \emph{left-orderable} if it admits a total order that is invariant under left multiplication.

\begin{rem}\label{rem:convexity}
 Since $M$ is cusped, $\pi_1(M)$ is left-orderable \cite[Theorem 1.1]{boyer-rolfsen-wiest}. The existence of an invariant total cusp order is equivalent to every peripheral subgroup being \emph{left relatively convex} in $\pi_1(M)$. Set $G=\pi_1(M)$. The condition that every peripheral subgroup of $G$ is left relatively convex has two equivalent formulations:
 \begin{itemize}
 \item For each peripheral subgroup $P\leq G$, there is a $G$-invariant total order on the coset space $G/P$.
 
     \item For each peripheral subgroup $P$, there is a left order on $G$ in which $P$ is convex.
     
 \end{itemize}
 The first characterisation is essentially the definition of a total cusp order. The second is equivalent to the first by \cite[Proposition 2.1.3]{basic_order_stuff}. Note the quantifier: we pick a different left order for each peripheral subgroup. 
\end{rem}
Since total cusp orders are related to certain left orders as in \Cref{rem:convexity}, it is also natural to discuss bi-orderability. A group is \emph{bi-orderable} if it admits a total order that is invariant under both left and right multiplication. 

A connection between bi-orderability and convex orders was made by Boyer-Clay while studying slope detection \cite{boyer-clay}. They observe a version of the following. 

\begin{lemma}\label{lem:bi-order_gives_cusp_order}
    Let $M$ be a cusped finite-volume hyperbolic three-manifold. Suppose that $\pi_1(M)$ is bi-orderable. Then there is a $\pi_1(M)$-invariant total order on $\cusps(\lift{M})$. 
\end{lemma}
This fact appears implicitly in \cite[Corollary 5.4, Example 7.1]{boyer-clay} and the argument uses \cite[Corollary 5.1.4]{basic_order_stuff}. We give a direct, self-contained proof that constructs the same order. 

\begin{proof}
    Fix a bi-order $\prec$ of $\pi_1(M)$. Let $P\simeq\bbZ^2$ be a peripheral subgroup of $\pi_1(M)$. We give a left-invariant order on the left cosets of $P$. Choose a nonidentity element $p\in P$. For $g,h\in \pi_1(M)$ define $<_p$ by:
    \[ gP <_p hP \text{ \quad if and only if \quad } gpg^{-1} \prec hp h^{-1}.\]
    
    We first claim that $<_p$ is a well-defined total order. If $hP= gP$, then $gpg^{-1}  = hph^{-1}$ since $P\simeq \bbZ^2$ is abelian. We must also show that for distinct cosets, the conjugates are distinct. Suppose that $gpg^{-1}  = hph^{-1}$. Then conjugation by $h^{-1}g$ fixes $p$. Since $P$ is a peripheral subgroup and $M$ is hyperbolic, $P$ is malnormal, and so $h^{-1}g\in P$. Hence $<_p$ is a well-defined total order.
    
     To show invariance, let $k\in \pi_1(M)$ and suppose that $gP<_p hP$. Since $\prec$ is a bi-order, it is preserved under conjugation, and so $kgP<_p khP$. 
     
    Repeat this for a peripheral subgroup representing each cusp orbit. Choose any order on the distinct cusp orbits. Together with the order on each cusp orbit, this induces a lexicographic invariant total order on $\Delta(\lift{M})$. 
    %
    %
\end{proof}

We obtain a connection between bi-orderability and branchings.

\begin{prop}\label{prop:bi-order_branching}
    Let $M$ be a cusped hyperbolic three-manifold. Let $\cal{T}$ be an essential ideal triangulation of $M$. If $\pi_1(M)$ is bi-orderable, then $\cal{T}$ is coorientably branchable. 
\end{prop}
\begin{proof}
    \Cref{lem:bi-order_gives_cusp_order} gives a cusp order. Apply \Cref{prop:cusps_give_branching}. 
\end{proof}
\begin{example}
    The fundamental groups of the figure-eight knot complement and the Whitehead link complement are known to be bi-orderable. The fundamental groups of their sibling manifolds are known to be non-bi-orderable \cite[Theorem 5.1, Theorem 5.2]{kin-rolfsen}.
    
    \Cref{prop:bi-order_branching} implies that all essential triangulations of the figure-eight knot complement and of the Whitehead link complement are coorientably branchable. So, the methods of the current paper do not rule out the presence of a veering triangulation on the Whitehead link complement, whereas \cite{misha-dunfield-egs,ss-veering-finiteness} both do. 
    
    On the other hand, the siblings have census triangulations that are essential and do not admit coorientable branchings. Thus \Cref{prop:bi-order_branching} gives an alternate proof of the non-bi-orderability result.
\end{example}

We give a census of examples in \Cref{subsec:census}. We finish our discussion of orders on $\pi_1(M)$ with circular orders.

\begin{rem}\label{rem:circles}
 There is also a convex-order characterisation for the invariant \emph{circular orders} on cusps constructed in \cite[\S 3]{ss-link}. However, any cusped hyperbolic three-manifold admits a circular order on its cusps, as follows. The Gabai-Mosher construction produces a blown-up pseudo-Anosov flow on $M$ \cite{mosher1996laminations, landry-tsang}. Then $\pi_1(M)$ acts on the associated bifoliated plane and its  boundary circle. This plane need not be loom. The cusps are points on this circle, giving a circular order. Schleimer-Segerman use these circular orders to count veering triangulations \cite{ss-veering-finiteness}. However, the ubiquity of the orders is one reason why their method applies to relatively few census manifolds.
\end{rem}



\section{Loom spaces and medians}\label{sec:medians}
We give a summary of the geometry of loom spaces as introduced in \cite{loom_new_arxiv} and then use this geometry to show that loom spaces have medians (\Cref{lem:medians}). We will use this structure to construct a total order in \Cref{sec:totalorder}. For background on bifoliated planes and pseudo-Anosov flows, see \cite{barthelme-mann-book}.

\subsection{Loom spaces}\label{subsec:loom}
Fix a finite veering triangulation $\cal{V}$ of $M$. Associated with $\cal{V}$ is the \emph{loom space} $\cal{L}$ introduced in \cite[Definition 2.11]{loom_new_arxiv}. This is a plane equipped with a pair of transverse foliations $\cal{L}^u$ and $\cal{L}^s$ that are preserved by a $\pi_1(M)$-action. We call $\cal{L}^u$ and $\cal{L}^s$ the \emph{unstable} and \emph{stable} foliations, respectively. Fix an orientation on each foliation. We use the south-to-north and west-to-east directions to refer to the orientations along $\cal{L}^s$ and $\cal{L}^u$ respectively.

An automorphism of $\cal{V}$ induces an automorphism of $\cal{L}$. The $\pi_1(M)$-action preserves the orientations on the foliations of $\cal{L}$ if and only if $\cal{V}$ is edge-orientable (\Cref{eg:vbs}).

A \emph{rectangle} $R\subset\cal{L}$ is the image of an embedding of $(0,1)^2$ into $\cal{L}$ that is foliated by $\cal{L}^s$ and $\cal{L}^u$ as a product. A \emph{cusp rectangle} is an embedding of $[0,1]^2$ with one corner removed, foliated as a product. See \Cref{fig:rectangles}. The sides of a cusp rectangle incident to the missing corner are called \emph{cusp sides}. Call two cusp rectangles \emph{adjacent} if one has a cusp side contained in a cusp side of the other. The \emph{cusps} of $\cal{L}$ are equivalence classes of cusp rectangles under the transitive closure of this adjacency relation. Let $\Delta(\cal{L})$ denote the set of cusps.

 By an abuse of terminology, we treat cusps as ideal points. To distinguish them, points in $\cal{L}$ are called \emph{material}. Indeed, $\cal{L}$ has a natural compactification to a disk \cite[\S 1.5 (A)]{ss-link} in which cusps are points along the boundary, though we do not use this structure.

The closure of a rectangle in $\cal{L}$ has four well-defined \emph{sides}. Each side can contain a cusp, and is thus a union of arcs in at most two leaves. We refer to these sides by their cardinal directions. Sides meet at the material or cusp \emph{corners} of the rectangle. 

\begin{figure}[h]
    \centering
    \includegraphics[width=0.5\linewidth]{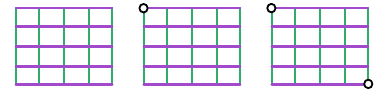}
    \caption{From left to right: a rectangle with all boundary points material, a cusp rectangle, and an edge rectangle. The foliations $\cal{L}^u$ and $\cal{L}^s$ are depicted in purple (horizontal) and green (vertical).}
    \label{fig:rectangles}
\end{figure}

The `loom' axioms dictate how rectangles behave with respect to cusps. We do not appeal directly to the definition, so we instead explicitly state the resulting properties that we use. Some of these can be intuited from the following.

\begin{rem}\label{rem:flow}
Concretely, the loom spaces we consider arise as follows. Since $\cal{V}$ is a veering triangulation, there are filling slopes on $M$ such that the filled manifold admits a \emph{pseudo-Anosov flow without perfect fits}. Let $\varphi$ be this flow \cite{agol-tsang,ss-pair}. Then the \emph{orbit space} $\cal{O}:=\lift{M}/\lift{\varphi}$ is a plane equipped with a pair of foliations and a $\pi_1(M)$-action \cite{fenley-mosher-qg}. Puncture the discrete set in $\cal{O}$ corresponding to the singular orbits of $\varphi$ and lift to the universal cover of the punctured plane. This gives another bifoliated plane. This plane is equivariantly isomorphic to $\cal{L}$ \cite[page 2]{ss-link}.
\end{rem}

For a point $p\in\cal{L}$, let $\cal{L}^s(p)$ and $\cal{L}^u(p)$ be the leaves of the respective foliations that contain $p$. A \emph{cusp leaf} of $\cal{L}$ is a leaf in either foliation with one end at a cusp. A leaf in $\cal{L}$ can only have a cusp at one end \cite[Lemma 3.10]{loom_new_arxiv}. For $p\in \Delta(\cal{L})$, let $\cal{L}^s(p)$ and $\cal{L}^u(p)$ be the unions of all the leaves in each respective foliation that have an end at $p$.

\begin{defn}\label{def:orthants}
    Let $p\in\cal{L}\cup \Delta(\cal{L})$. The connected components of $\cal{L} - (\cal{L}^u(p)\cup \cal{L}^s(p))$ are the \emph{orthants} of $p$. A \emph{closed orthant} is the union of an orthant with its two boundary leaves.
\end{defn}

The orientations on the foliations give an orthant a pair of cardinal directions. So, we speak, for example, of a northeastern orthant. A material point has four orthants. The loom property implies that the orthants of a cusp are indexed by $\bbZ$, with their order induced by the adjacency relation \cite[Lemma 6.15]{loom_new_arxiv}. Ascending in the order on the orthants corresponds to walking around the puncture in the drilled orbit space of \Cref{rem:flow}. We thus think of walking around a cusp as analogous to walking around a branch cut.

The Agol-Gu\'eritaud construction builds $\cal{V}$ explicitly from $\cal{L}$ (\cite[\S 6]{loom_new_arxiv}, \cite[\S 3]{landry-minsky-taylor}). The cusps of $\cal{L}$ are in $\pi_1(M)$-equivariant correspondence with the vertices of $\lift{\cal{V}}$, the lift of $\cal{V}$ to the universal cover. This gives a correspondence between $\cusps(\cal{L})$ and $\cusps(\lift{M})$. An \emph{edge rectangle} in $\cal{L}$ is an embedding of $[0,1]^2$ with two opposite corners removed, foliated as a product. Two vertices of $\lift{\cal{V}}$ span an edge if and only if their corresponding cusps in $\cusps(\cal{L})$ bound an edge rectangle in $\cal{L}$. See \Cref{fig:rectangles}.

\subsection{Geodesics and hulls}
We use the theory of geodesics and hulls of loom spaces as detailed in \cite[\S4]{loom_new_arxiv}. We add some additional deductions needed later.

A \emph{polygonal path} $\gamma$ in $\cal{L} \cup \cusps(\cal{L})$ is an oriented path given by a union of arcs in a leaf of either foliation that are glued end-to-end. The maximal connected arcs of $\gamma$ in a fixed leaf are the \emph{stable} and \emph{unstable edges}. These are connected at the \emph{vertices} of $\gamma$. A vertex can be a material point or a cusp.

\begin{defn}\label{def:geodesic}
    A polygonal path in $\cal{L}\cup\cusps(\cal{L})$ is a \emph{geodesic} if its intersection with each leaf of both $\cal{L}^s$ and $\cal{L}^u$ is connected.  
\end{defn}

A \emph{turn} of $\gamma$ is a material vertex $v$ of $\gamma$. At $v$, there is a small rectangle $R$ with $v$ as a corner such that $\gamma$ contains two sides of $R$. We call a turn \emph{left} or \emph{right} if $R$ is locally to the left or right of $\gamma$. We say that $\gamma$ makes a \emph{U-turn} if it has three consecutive edges that are sides of a common rectangle in $\cal{L}$. A polygonal path is geodesic if and only if it has no U-turns \cite[Lemma 6.10]{loom_new_arxiv}.

\begin{figure}[h]
    \centering
    \includegraphics[width=0.6\linewidth]{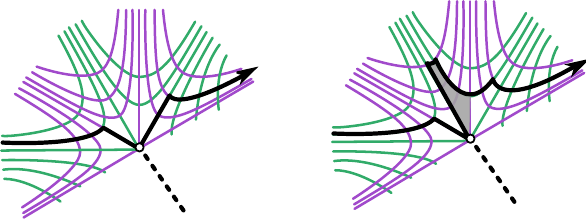}
    \caption[A geodesic in a loom space.]{Left: A geodesic in $\cal{L}$ as it passes through a cusp, winding around two orthants. A dotted line represents the `branch cut'. The orthants of the cusp wind around either side of this cut. Right: A polygonal path that has a U-turn shaded in grey. One can eliminate U-turns one-by-one in this path to obtain the geodesic on the left.}
    \label{fig:geodesic}
\end{figure}

\begin{defn}\label{def:between}
 For $p,\,q,\,r\in \cal{L}\cup\cusps(\cal{L})$, say that $q$ is \emph{orthant-between} $p$ and $r$ if $p$ and $r$ lie in non-adjacent closed orthants of $q$. Say that $q$ is \emph{between} $p$ and $r$ if in every neighbourhood of $q$, there is a point that is orthant-between $p$ and $r$.
\end{defn}

We note that \Cref{def:between} is the definition from the recent corrected version \cite[Definition 6.16]{loom_new_arxiv}.

\begin{lemma}\label{lem:omni_orthants}
    Suppose that $p,q,r \in\cal{L}\cup\Delta(\cal{L})$.
    \begin{enumerate}[label=\roman*)]
        \item All geodesics from $p$ to $q$ are contained in a common closed orthant of $p$ and a common closed orthant of $q$.
        
        \item There is a geodesic from $p$ to $r$ through $q$ if and only if $q$ is between $p$ and $r$.

        \item If $q$ is between $p$ and $r$, then both $q$ and $r$ lie in a common closed orthant of $p$. 
    \end{enumerate}
\end{lemma}
\begin{proof}
    For i), suppose first that $q$ lies on the interior of a closed orthant of $p$. Then any geodesic from $p$ to $q$ must lie in the closure of this orthant, or it would be forced to take a U-turn. If $q$ lies on a leaf through $p$, then all geodesics are leaf arcs, and the claim holds.
    
    Item ii) is \cite[Lemma 6.17]{loom_new_arxiv}.
    
    Finally, suppose that $q$ is between $p$ and $r$. Take a geodesic from $p$ to $r$ through $q$, given by ii). All such geodesics lie in a common orthant of $p$ by i). So $q$ lies in this same orthant of $p$, giving iii). 
\end{proof}

\begin{defn}\label{def:hulls}
    For $p,\,q\in\cal{L}\cup \cusps(\cal{L})$, the \emph{convex hull} $H(p,q)\subset \cal{L}\cup \cusps(\cal{L})$ is the set of points between $p$ and $q$.
\end{defn}

\begin{lemma}\label{lem:geodesics}
    Let $p,\,q\in \cal{L}\cup \cusps(\cal{L})$. Then:
    \begin{enumerate}[label=\roman*)]
        \item $H(p,q)$ is the union of all geodesics from $p$ to $q$.
        \item $\partial H(p,q)$ is the union of the unique leftmost and unique rightmost geodesics. The leftmost geodesic takes only right turns, and the rightmost geodesic only left turns. 
    \end{enumerate}
\end{lemma}
\begin{proof}
    These are proved in \cite[Lemma 6.23]{loom_new_arxiv}.
\end{proof}

\begin{lemma}\label{lem:hulls}
    Let $p,\,q\in \cusps(\cal{L})$. Then the unique leftmost and unique rightmost geodesics forming $\partial H(p,q)$ meet at a unique ordered tuple $p=c_0, c_1,\dots c_n = q$ for $c_i\in \cusps(\cal{L})$. Moreover, there are finite-sided disks $(P_i)_{i=0}^{n-1}\subset \cal{L}\cup \cusps(\cal{L})$ such that:
    \begin{enumerate}[label=\roman*)]
        \item $H(p,q) = \cup_i P_i$.
        
        \item We have  $P_i\cap P_{i+1} = \{c_{i+1}\}$ and  $P_i\cap P_j$ is empty for $\vert i-j\vert>1$.
        
        \item $P_i = H(c_i,c_{i+1})$.

        \item For each $i$, any geodesic from $c_i$ to $c_{i+1}$ has edges pointing in a common direction along each foliation.
        \item $P_i$ and $P_{i+1}$ lie in distinct, non-adjacent closed orthants of $c_{i+1}$. Moreover, $P_i$ and $P_{i+1}$ do not share an adjacent closed orthant. 
    \end{enumerate}
        %
        %
        %
        %
        
    For $p,q\in \cal{L}\cup \cusps(\cal{L})$ not necessarily cusps, we obtain a similar decomposition, now allowing $P_0$ and $P_{n-1}$ to be single leaf segments. 
\end{lemma}

\begin{proof}
We prove the claim for $p,q\in\cusps(\cal{L})$. The material case is similar. 
 In the proof of Item 4 of \cite[Lemma 6.23]{loom_new_arxiv}, they show that the leftmost and rightmost geodesics meet at a finite set of cusps $c_i$, and that the hulls between these cusps are finite-sided disks. Let the $P_i$ be these disks, ordered as they are traversed by one of the boundary geodesics. They also show that these disks have disjoint sets of material points.
 
 This setup gives items  i), ii) and iii). 
    %
    %
    %

    Let $\gamma$ be a geodesic in the interior of a given $P_i$. Then a change of direction in the (un)stable edges of $\gamma$ would force a U-turn. For a geodesic in $P_i$ that meets the boundary, we can push an edge off of the boundary across a small rectangle to obtain a geodesic in the interior of $P_i$. This does not change the direction along any edges, and so we obtain iv).

    For v), suppose that $P_i$ and $P_{i+1}$ lie in either adjacent closed orthants or the same closed orthant of $c_{i+1}$. Then $P_i\cap P_{i+1}$ contains material points, contradicting iii). Suppose then that $P_i$ and $P_{i+1}$ have a common adjacent orthant. Then there is a small rectangle in this orthant with $c_{i+1}$ as a corner. We can reroute either the leftmost or rightmost geodesic from $p$ to $q$ across $R$. This gives another geodesic that is further left or further right, a contradiction. 
\end{proof}

\begin{figure}[h]
    \centering
    \includegraphics[]{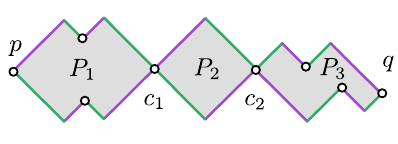}
    \caption[A convex hull.]{The structure of a convex hull between $p$ and $q$ given by \Cref{lem:hulls}.}
    \label{fig:hulls}
\end{figure}

See \Cref{fig:hulls}. We call a polygon that is the hull of two points or cusps a \emph{staircase polygon}.

For $i\neq0, n$, we call $c_i$ a \emph{cinch cusp} of $H(p,q)$. If either or both of $c_0=p$ and $c_n=q$ are cusps, we also call these cinch cusps. We also use $H(p,q)$ to encode the natural direction on the hull from $p$ to $q$. That is, $H(q,p)$ is the same underlying set as $H(p,q)$ but with the tuples $(P_i)$ and $(c_i)$ reversed.

\begin{lemma}\label{lem:geodesic_concat}
    For $p\in\cal{L}\cup\cusps(\cal{L})$, let $\alpha$ and $\beta$ be geodesics that end and start at $p$, respectively. Suppose that $\alpha$ and $\beta$ are contained in non-adjacent closed orthants of $p$. Then $\alpha$ concatenated with $\beta$ is a geodesic.
\end{lemma}
\begin{proof}
    A leaf that meets one closed orthant of $p$ can only extend into an adjacent closed orthant. Thus, $\alpha$ and $\beta$ do not intersect a common leaf. Hence their concatenation is a geodesic.
\end{proof}
\begin{lemma}\label{lem:follow_geodesics}
    Let $p,q\in \cal{L}\cup \cusps(\cal{L})$. Let $x\in H(p,q)$, and suppose that $\alpha$ is a geodesic from $p$ to $x$ and $\beta$ a geodesic from $x$ to $q$. Then concatenating $\alpha$ with $\beta$ is a geodesic from $p$ to $q$. 
\end{lemma}
\begin{proof}
    Since $x$ is between $p$ and $q$, there is a geodesic $\gamma$ from $p$ to $q$ through $x$. Let $\gamma_p$ and $\gamma_q$ be the two halves of $\gamma$ that end and start at $x$, respectively. Both $\gamma_p$ and $\alpha$ lie in $H(p,x)$. For each polygon $P_i$ in the chain forming $H(p,x)$, the staircase structure implies that both $\gamma_p$ and $\alpha$ meet every non-boundary leaf of $P_i$. Thus, they meet the same set of leaves except for finitely many leaves containing the edges of $\partial H(p,x)$. Similarly, this is true of $\gamma_q$ and $\beta$. If $\alpha\beta$ is not a geodesic, it contains a U-turn, and thus meets infinitely many leaves in two distinct components. Then $\gamma = \gamma_p\gamma_q$ also meets all but finitely many such leaves in two distinct components, contradicting that $\gamma$ is a geodesic. 
    %
    %
\end{proof}

\subsection{Medians}
Given a triple of points in $\cal{L}\cup\cusps(\cal{L})$, a \emph{median} is a point between every pair of points in the triple. We now prove a technical lemma: cusps of loom spaces have medians. The proof breaks into several cases. We isolate an easy one.

\begin{claim}\label{claim:easy_median}
    Let $p,\,q,\,r\in  \cusps(\cal{L})$ be distinct. If $q$ is between $p$ and $r$, then
    \[
    H(p,q)\cap H(q,r)\cap H(r,p) = \{q\}.
    \]
\end{claim}
\begin{proof}
    Since the points are all cusps, $q$ is orthant-between $p$ and $r$, so $q\in H(p,r)$. Since $H(p,q)$ and $H(q,r)$ are contained in non-adjacent closed orthants of $q$, we have $H(p,q)\cap H(q,r) = \{q\}$. So, the triple intersection is $\{q\}$. 
\end{proof}
For the general case, most of the work is in the following local treatment. 

\begin{lemma}\label{lem:local_median}
    Let $a,b,c\in\Delta(\cal{L})$ be distinct. Suppose that $H(b,a)$ and $H(b,c)$ each consist of one polygon. Then there is $x\in H(a,b)\cap H(b,c)\cap H(c,a)$.
    %
\end{lemma}
\begin{proof}
    We break into three successively longer cases depending on which closed orthants of $b$ contain the hulls $H(b,a)$ and $H(b,c)$.

    \begin{figure}[h]
        \centering
        \includegraphics[width=0.7\linewidth]{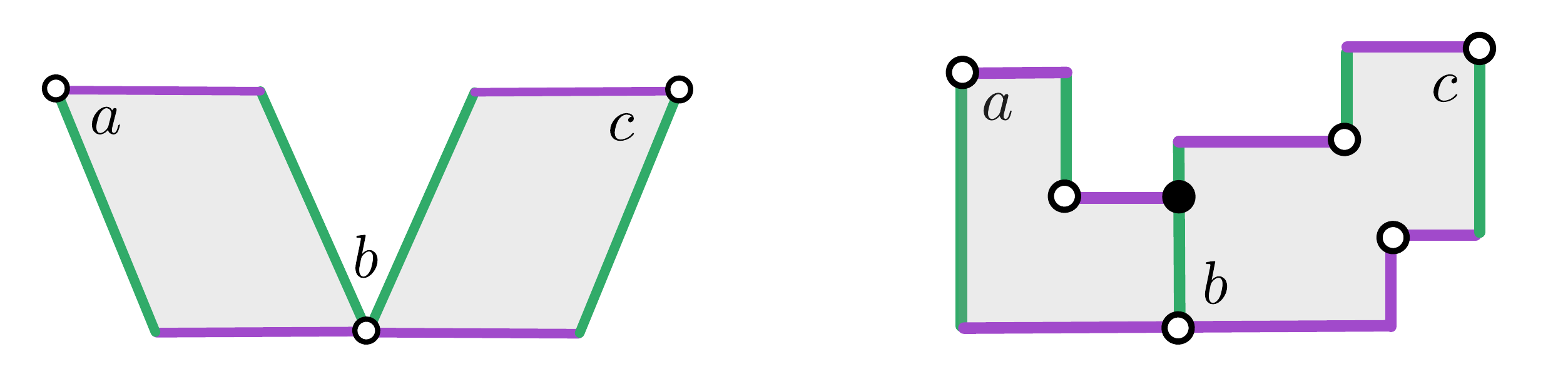}
        \caption[Local medians in distinct orthants.]{Case A (left) and Case B (right) of the local medians.}
        \label{fig:meds_first}
    \end{figure}
    \textbf{Case A:}  Suppose that $H(b,a)$ and $H(b,c)$ are in non-adjacent closed orthants of $b$. Then $b$ is between $a$ and $c$. Apply \Cref{claim:easy_median}. 

    \textbf{Case B:} Suppose that $H(b,a)$ and $H(b,c)$ are in adjacent closed orthants of $b$. Then $H(b,a)\cap H(b,c)$ is a single stable or unstable edge. Breaking symmetry, suppose it is a stable arc in a leaf $l\in \cal{L}^s$ directed northward from $b$. Since leaves in $\cal{L}$ do not contain more than one cusp, we may assume that $H(b,c)$ extends farther north along $l$ than $H(b,a)$. We may also assume that $H(b,c)$ lies in a northeastern orthant of $b$.
    
    Let $x$ be the northernmost point on this arc. Then $x\in H(a,b)\cap H(b,c)$. See \Cref{fig:meds_first}, right. Let $y\in l$ lie just north of $x$. Then $y\in H(b,c)$, and $c$ lies in the northeastern orthant of $y$.
    
    Let $r_a$ be the rightmost geodesic of $H(b,a)$. Then $r_a$ passes through $x$ and begins in the southwestern orthant of $y$. For $r_a$ to leave this orthant, it must cross the western half of a leaf $j = \cal{L}^u(y)$. Suppose that it does. Then $r_a$, $l$ and $j$ bound a staircase polygon. We can then reroute $r_a$ over $l$ and $j$ to give a geodesic from $b$ to $a$ that is to the right of $r_a$. This contradicts that $r_a$ is rightmost. So, $a$ lies in the southwestern orthant of $y$. Thus $y$ is orthant-between $a$ and $c$. So $x\in H(a,c)$.
    
    \textbf{Case C:} Suppose that $H(b,a)$ and $H(b,c)$ lie in the same closed orthant of $b$. Say it is a northeastern orthant.  Let $l_a$ and $r_a$ be the leftmost and rightmost geodesics bounding $H(b,a)$ (\Cref{lem:hulls}). Define similarly $l_c$ and $r_c$ bounding $H(b,c)$. Following $l_c$ and $l_a$, one of the two must either terminate or turn away from the other. Suppose it is $l_c$. If $l_c$ terminates where $l_a$ continues, then $c\in H(b,a)$. We can then take $x=c$ and be done. So, we may assume that $l_c$ turns to the right of $l_a$. 

    If $l_a$ and $l_c$ intersect after diverging, there is a small rectangle to the left of $l_c$ at this next intersection. Rerouting $l_c$ over this rectangle gives a geodesic from $b$ to $c$ that lies to the left of $l_c$, a contradiction. Since $l_c$ cannot terminate in the interior of $H(a,b)$, we deduce that $l_c$ intersects $r_a$ at some point $x$. Then $x\in H(a,b)\cap H(b,c)$. 
    
    
    If $x \in \{c,a\}$, then $x\in H(c,a)$ and we are done. So assume $x\not\in\{a,c\}$. Now there are two cases:

    \begin{itemize}
        \item Suppose that $x$ is material. If a northward edge of $l_c$ meets an eastward edge of $r_a$, then $l_c$ is to the right of $r_a$. This contradicts that $r_a$ is rightmost from $b$ to $a$. Thus $x$ is the intersection of an eastward edge of $l_c$ and a northward edge of $r_a$. See \Cref{fig:meds_same}, left.
        
        Now $l_c$ and $r_a$ must cross, since otherwise $l_c$ would make a left turn at $x$. Let $y$ be a point close to $x$ lying in the northeastern orthant of $x$. Once $r_a$ leaves a small neighbourhood of $y$, it enters the northwestern orthant of $y$. Arguing similarly to Case B, $r_a$ must remain in this orthant. Similarly, $l_c$ enters and remains in the southeastern orthant of $y$. So, $a$ and $c$ lie in non-adjacent closed orthants of $y$. Thus $y$ is orthant-between $a$ and $c$.  Therefore $x\in H(a,c)$.
        
        \item Suppose that $x$ is a cusp. Then $l_c$ and $r_a$ approach $x$ along distinct leaves that bound an orthant of $x$. See \Cref{fig:meds_same}, right. Since $x\neq a, c$, the hulls $H(b,c)$ and $H(b,a)$ continue beyond $x$ into distinct non-adjacent orthants of $x$. So $x$ is orthant-between $a$ and $c$, giving $x\in H(a,c)$.\qedhere
    \end{itemize}
    
    \begin{figure}[h]
        \centering
        \includegraphics[width=0.6\linewidth]{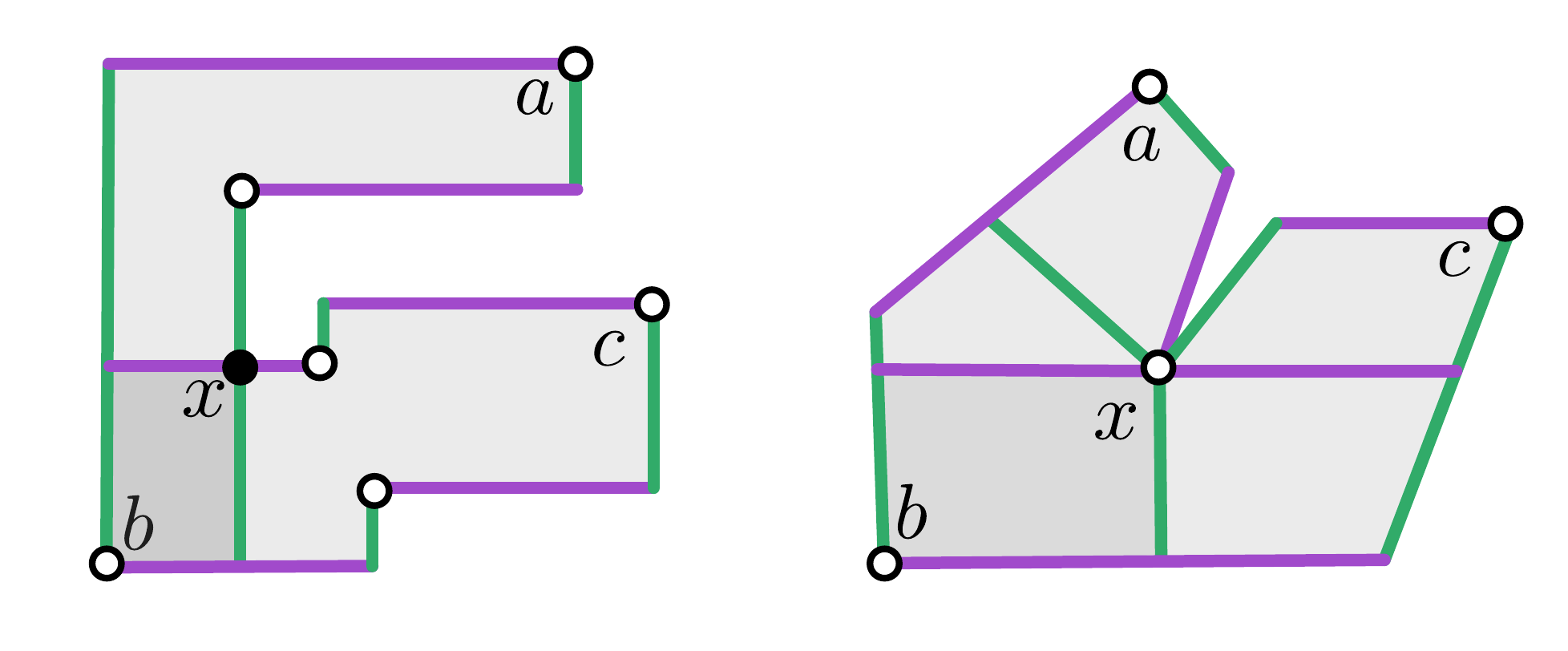}
        \caption[Local medians in the same orthant.]{The two cases of the local medians in the same orthant of $b$. Left: the material case. Right: the cusp case.}
        \label{fig:meds_same}
    \end{figure}
    
\end{proof}

\begin{lemma}\label{lem:medians}
    Let $p,\,q,\,r\in  \cusps(\cal{L})$ be distinct. Then
    \[
    H(p,q)\cap H(q,r)\cap H(r,p) \neq \varnothing. 
    \]
\end{lemma}

\begin{proof}
    Let $p,q,r\in\cusps(\cal{L})$. If one of the three cusps lies between the other two, apply \Cref{claim:easy_median}. So suppose otherwise.
    
    Suppose that $d\in\cusps(\cal{L})$ is a cinch cusp of both $H(q,p)$ and $H(q,r)$. By \Cref{lem:hulls}, both $H(q,p)$ and $H(q,r)$ start with the same initial chain of polygons $H(q,d)$. Let $b\in\cusps(\cal{L})$ be the last cinch cusp of $H(q,p)$ such that $b$ is also a cinch cusp of $H(q,r)$. This includes the possibility that $b=q$.
     
     Let $a$ and $c$ be the two distinct cinch cusps along $H(q,p)$ and $H(q,r)$ immediately after $b$, respectively. This includes the possibilities $a=p$ and $c=r$. The hypothesis that no one of $p,\,q,\,r$ lies between the other two implies that $a,\,b,\,c$ are distinct. 
    
    We further suppose that neither $H(b,a)$ nor $H(b,c)$ contains the other. Note that $H(b,a)$ and $H(b,c)$ by definition consist of single polygons. So, we can apply \Cref{lem:local_median} to obtain a local median $x \in H(a,b) \cap H(b,c) \cap H(c,a)$. Then $x\in H(p,q)\cap H(q,r)$, so we must show that $x \in H(r,p)$. 

    Let $\gamma$ be a geodesic from $c$ to $a$ through $x$.  Since $a$ is a cinch cusp, \Cref{lem:hulls} Item v)  implies that $\gamma$ and $p$ lie in distinct non-adjacent closed orthants of $a$. By \Cref{lem:geodesic_concat}, we can extend $\gamma$ to a geodesic from $c$ to $p$. We can extend $\gamma$ similarly from $c$ to $r$. This gives a geodesic between $p$ and $r$ containing $x$. Thus $x \in H(r,p)$.

    Now suppose that one of $H(b,a)$ and $H(b,c)$ contains the other. Say that $H(b,c)\subset H(b,a)$. If $a$ and $H(c,r)$ lie in non-adjacent closed orthants of $c$, then we can choose the same median $x$ as in the non-containment case and repeat the concatenation argument (\Cref{fig:bad_meds}, left). Otherwise, $H(c,r)$ and $a$ lie in either the same closed orthant or adjacent closed orthants of $c$. The same-orthant case is precluded by \Cref{lem:hulls} Item v) since $c$ is a cinch cusp.
    
    So, suppose that $H(c,r)$ and $a$ are in adjacent orthants. See \Cref{fig:bad_meds}, right. Let $c'$ be the next cinch cusp along $H(c,r)$. The hull $H(c,c')$ is a single polygon by definition. Moreover, $H(c,a)$ is a single polygon by \Cref{lem:hulls} Item v). Thus, $c,\,a,\,c'$ have the same configuration as $b,\,a,\,c$ in the non-containment case. Specifically, it is the setting of the second panel of \Cref{fig:good_meds}. So we can apply \Cref{lem:local_median} to obtain a local median $x'$. The same argument as in the non-containment case gives $x'\in H(p,q)\cap H(q,r)\cap H(r,p)$. \qedhere
    %
%
    \begin{figure}[h]
        \centering
        \includegraphics[width=\linewidth]{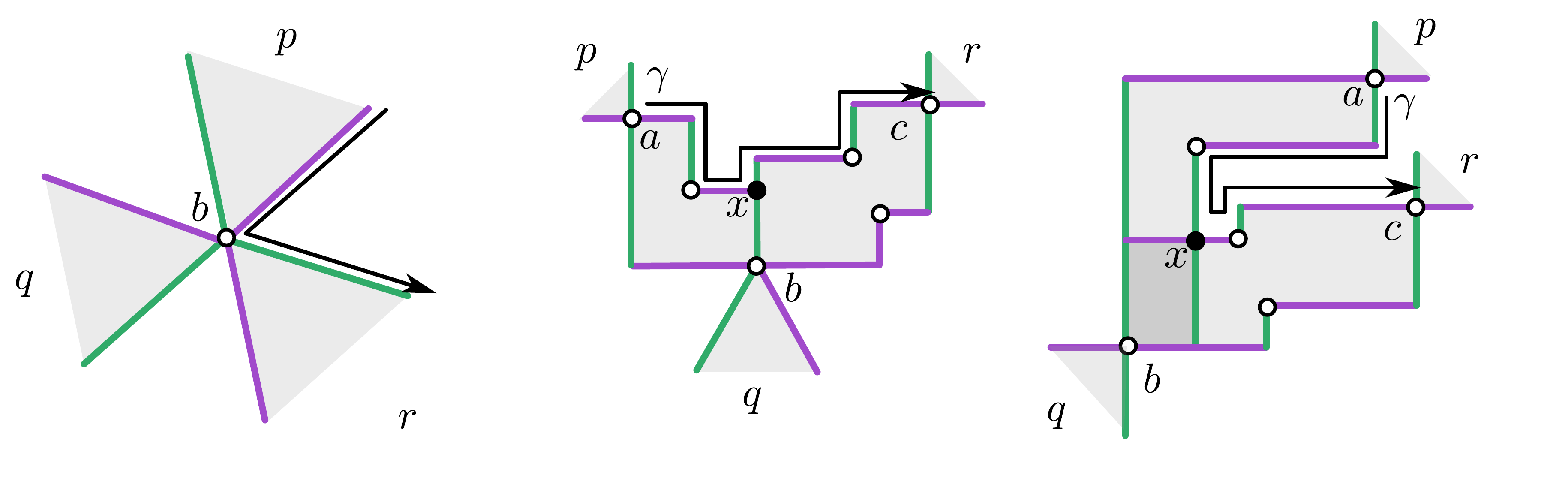}
        \caption[Global medians.]{The arrangement of the cusps when finding medians. The local configurations correspond to cases A, B and C of \Cref{lem:local_median} from left to right.}
        \label{fig:good_meds}
    \end{figure}
\end{proof}

\begin{figure}[h]
\centering
\includegraphics[width=0.6\linewidth]{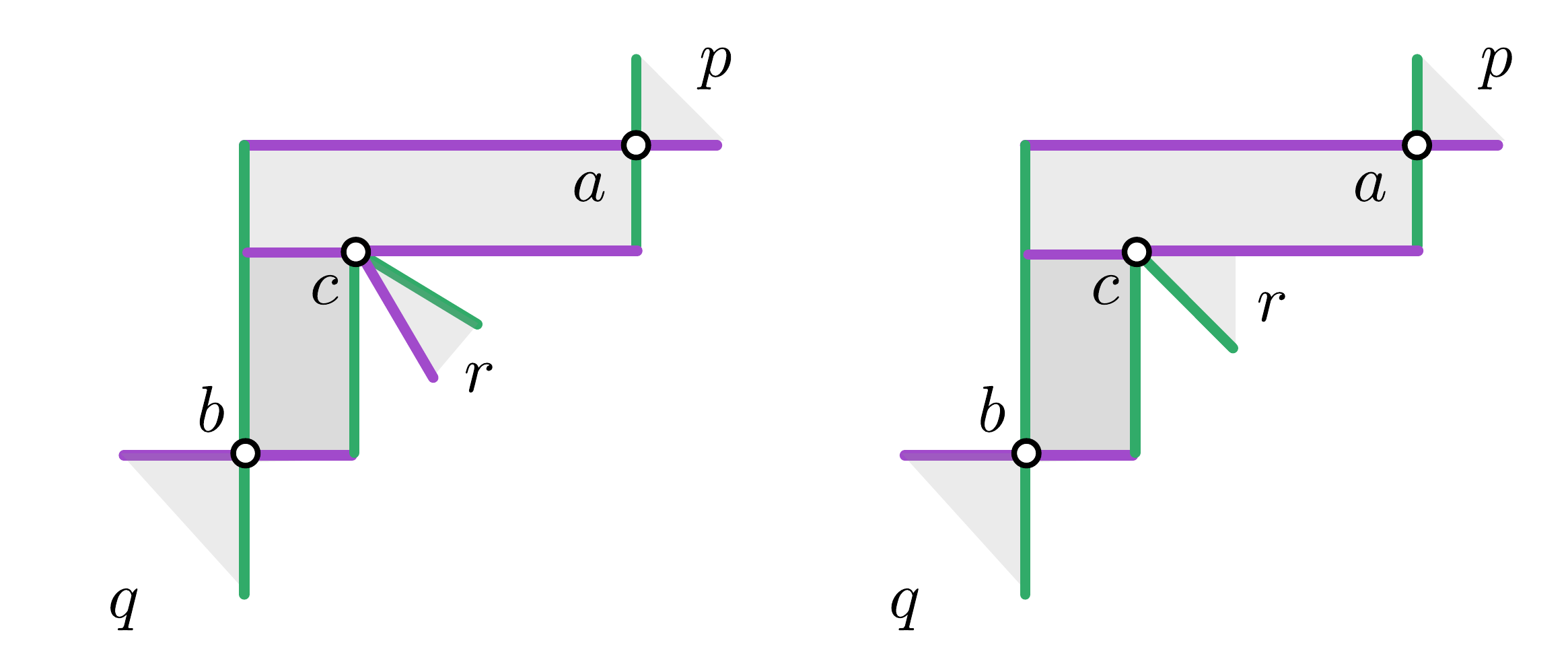}
\caption[Medians in the containment case.]{The two positions of $H(p,r)$ when $H(b,c)\subset H(b,a)$.}
\label{fig:bad_meds}
\end{figure}
    
\begin{rem}
    One can prove that the medians we find are unique. In this sense, loom spaces have a tree-like structure. Proving uniqueness requires an additional argument at each step. Since we do not require it for our main results, we omit the proof. 
\end{rem}

\section{Total cusp orders}\label{sec:totalorder}
Let $\cal{V}$ be a veering triangulation of $M$. Let $\cal{L}$ be the loom space corresponding to $\cal{V}$. Then $\cusps(\cal{L})$ comes equipped with natural partial orders. Namely, suppose that $p,q\in\cusps(\cal{L})$ are opposite corners of an edge rectangle. Then one of $p$ and $q$ is east-most with respect to the orientation on $\cal{L}^u$. The transitive closure of this relation is a partial order on $\cusps(\cal{L})$ \cite[Lemma 10.30]{ss-link}. 

In this section, we upgrade this partial order to a total order, proving \Cref{thm:totalorder}. We then use this to deduce the main criterion, \Cref{thm:smoothcriterion}. 

\subsection{Geodesic turning numbers}
 Suppose for this subsection that $\pi_1(M)$ preserves the orientations on the stable and unstable foliations of $\cal{L}$. We do not relax this assumption until \Cref{subsec:vts_give_branchings}.  

Let $p,q\in \cal{L}\cup\cusps(\cal{L})$. Let $\gamma$ be a geodesic between $p$ and $q$, and suppose that $\gamma$ has at least one unstable edge. Let $(c_i)$ and $(P_i)$ be the cinch cusps and polygons of the convex hull $H(p,q)$ given by \Cref{lem:hulls}. Suppose that $\gamma$ traverses $P_i$ and $P_{i+1}$ in opposite unstable directions. We call the cusp $c_{i+1}$ that lies between $P_i$ and $P_{i+1}$ a \emph{turning cusp} for $\gamma$.

Let $(t_j)$ be the tuple of turning cusps of $\gamma$ for $j=1,\dots,k$. Let $(\gamma_i)$ be the subgeodesics obtained by subdividing $\gamma$ at the turning cusps. Each unstable edge of $\gamma$ has an associated sign $-1$ or $+1$ according to whether it points westward or eastward, respectively. Each $\gamma_i$ has a sign $\sigma(\gamma_i)\in \{-1,1\}$ by taking the sign of any of its unstable edges.

Using the orientation on $\cal{L}$, the orthants at a cusp are ordered in the counterclockwise or clockwise direction. By \Cref{lem:hulls} Item v), $\gamma$ continues into a non-adjacent orthant at a turning cusp $t_j$. This outgoing orthant is either counterclockwise or clockwise from the incoming orthant at $t_j$. Define $\hat{\omega}_{t_j}(\gamma)$ as $+1$ or $-1$, respectively. Call $\hat{\omega}_{t_j}$ the \emph{untwisted winding sign at $t_j$}. We do not count how many orthants $\gamma$ winds around these cusps. 

Define
\[
\omega_{t_j}(\gamma) = \sigma(\gamma_{j-1})\hat{\omega}_{t_j}(\gamma).
\]
 Call ${\omega}_{t_j}$ the \emph{twisted winding sign at $t_j$}. The preference for $\omega$ over $\hat{\omega}$ will become evident: see \Cref{rem:reason_for_twisted}.

\begin{defn}\label{def:turning_number}
   Define the \emph{twisted turning number} of $\gamma$ as:
\begin{equation}
\gw(\gamma) := \sum_{j=1}^{k} \omega_{t_j}(\gamma) +\sum_{i=0}^{k}\sigma(\gamma_i). 
\label{eqn:turning_number}
\end{equation}

For $p,q\in\cusps(\cal{L})$ distinct cusps, set $\gw(p,q): = \gw(\gamma)$ for any geodesic $\gamma$ between $p$ and $q$.
\end{defn}

See \Cref{fig:turning_number}.
\begin{figure}[h]
    \centering
    \includegraphics[width=0.5\linewidth]{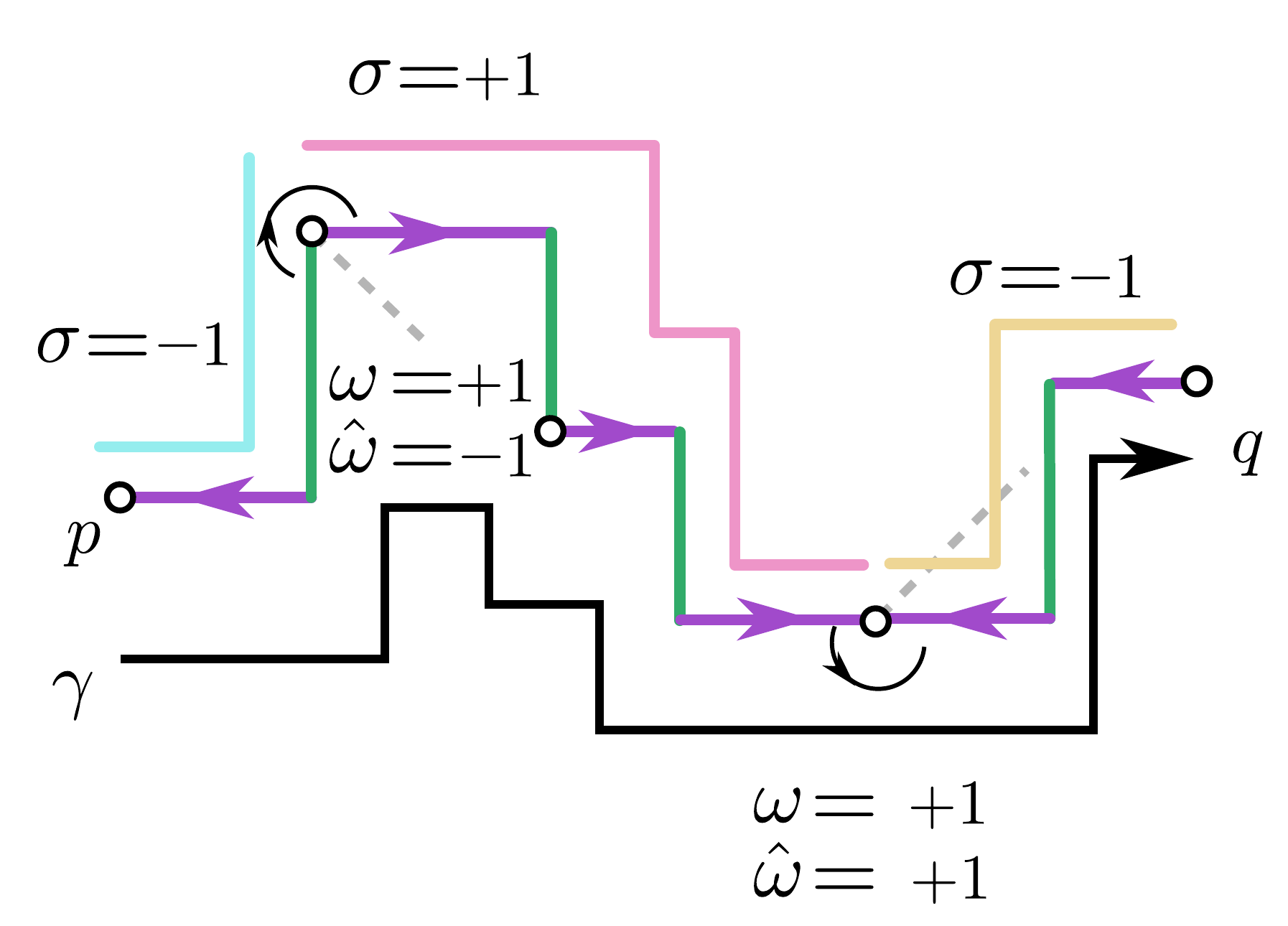}
    \caption[A sample computation of $\gw$.]{Computing the turning number of a geodesic $\gamma$. Here, $\Theta(\gamma)=+1$. Neither winding at a cusp (hollow dot) without changing direction nor turning at a material point contributes to the sum.}
    \label{fig:turning_number}
\end{figure}

\begin{lemma}
    The turning number $\gw(p,q)$ is well-defined.
\end{lemma}
\begin{proof}
    We must show that $\gw(p,q)$ does not depend on the choice of geodesic. Let $\gamma,\, \beta$ be two geodesics from $p$ to $q$. The decomposition into subgeodesics is determined entirely by the structure of $H(p,q)$. By \Cref{lem:hulls} Item iv), any geodesic from $p$ to $q$ crosses each polygon $P_i$ with unstable edges pointing the same direction. So, $\gamma$ and $\beta$ have the same turning cusps. The winding directions of $\gamma$ and $\beta$ agree since they are determined by the orthants of $c_{i+1}$ that contain $P_i$ and $P_{i+1}$.  So $\gw(\gamma)=\gw(\beta)$. 
\end{proof}

\begin{rem}
Note that the terms of the $\sigma$-sum in \Cref{def:turning_number} alternate, so the sum takes values in $\{-1\,,0,\,1\}$. Thus, this term determines the sign of $\gw(\gamma)$ precisely when the winding term is zero.
\end{rem}

\begin{lemma}\label{lem:order_basics}
    Suppose that the $\pi_1(M)$-action preserves the orientations on the foliations of $\cal{L}$. For $p,q\in\cusps(\cal{L})$ distinct, we have:
    \begin{itemize}
        \item $\gw(p,q) \neq 0$,
        \item $\gw(p,q) = -\gw(q,p)$,
        \item $\gw(g\cdot p,g\cdot q) = \gw(p,q)$ for all $g\in \pi_1(M)$,
        \item if $p$ and $q$ lie on the west and east sides of a common edge rectangle in $\cal{L}$, then $\gw(p,q)=1$. 
    \end{itemize}
\end{lemma}

\begin{proof}
    For a geodesic $\gamma$, if it meets $k$ turning cusps, it decomposes into $k+1$ subgeodesics. This implies that $\gw(p,q)$ is a sum of an odd number of $\pm1$. Hence $\gw(p,q) \neq 0$.
    
    Swapping $p$ and $q$ reverses the indexing on the hulls. So, taking the reverse of the geodesic, the subgeodesic chains and their turning cusps are the same with reversed indexing. Thus the $\sigma$ terms are negated. The winding direction $\hat{\omega}$ at each turning cusp is reversed. The sign of a subgeodesic as it enters a given turning cusp is unchanged: we approach the opposite side of the cinch cusp in the opposite direction. So, the $\omega$-terms are also negated. This gives $\gw(p,q) = -\gw(q,p)$.
    
    Invariance follows immediately from the invariance of $\hat{\omega}$ and $\sigma$.
    
    For the final item, we take a single geodesic from $p$ to $q$ inside the edge rectangle. Its subgeodesic decomposition is then the geodesic itself pointing in the unstable direction, giving $\gw(p,q)=1$. 
\end{proof}



\subsection{Turning numbers to a total order}
To use $\Theta$ to define a total order, the remaining work after \Cref{lem:order_basics} is to prove transitivity. For this, we use the median given by \Cref{lem:medians}. 

Let $\alpha$ be a geodesic that contains at least one unstable edge and terminates at $x$. Let $\sigma_x(\alpha)$ denote the incoming unstable sign of $\alpha$ at $x$.
\begin{lemma}\label{lem:prewind_identity}
    Let $\alpha$ and $\beta$ be geodesics with at least one unstable edge. Suppose that they both terminate at $x\in \cal{L}\cup \cusps(\cal{L})$. Suppose that $\alpha\beta^{-1}$ is a geodesic. Then
    \[
\gw(\alpha\beta^{-1}) =
\begin{cases}
\gw(\alpha)-\gw(\beta)+\omega_x(\alpha\beta^{-1}), & \sigma_x(\alpha)=\sigma_x(\beta), \\[4pt]
\gw(\alpha)-\gw(\beta)-\sigma_x(\alpha), & \sigma_x(\alpha)=-\sigma_x(\beta).
\end{cases}
\]
\end{lemma}
\begin{proof}
    If $\sigma_x(\alpha) = \sigma_x(\beta)$, then $x$ is a turning cusp of $\alpha\beta^{-1}$. The geodesic subdivision of $\alpha \beta^{-1}$ is then just the union of their subdivisions, and we obtain the additional turning term at $x$.

    If instead $\sigma_x(\alpha) = -\sigma_x(\beta)$, then $\alpha\beta^{-1}$ does not turn at $x$. Then the start of $\beta^{-1}$ and the end of $\alpha$ lie in the same subgeodesic in the decomposition of $\alpha\beta^{-1}$. Thus this $\sigma$-term for $\Theta(\alpha\beta^{-1})$ is counted once in each of $\Theta(\alpha)$ and $\Theta(\beta^{-1})$. Correcting this gives the second case. 
\end{proof}

Let $p,q,r\in\Delta(\cal{L})$. Loosely, \Cref{lem:medians} and \Cref{lem:prewind_identity} together say that the turning number $\Theta(p,r)$ differs from $\Theta(p,q)$+$\Theta(q,r)$ by a sum of local defects in \Cref{lem:prewind_identity}. We frame this as follows.
\begin{lemma}\label{lem:cycle_identity}
    Let $p,\,q,\,r\in \cusps(\cal{L})$ be distinct. Then:
    \begin{equation}
    \vert \gw(p,q)+\gw(q,r)+\gw(r,p)\vert = 1 \label{eq:cycle_identity}
    \end{equation}
\end{lemma}

We break the proof of the identity into three cases which we prove as separate claims. For the remainder of the subsection, fix a median $x$ between $p,\,q,\,r$ given by \Cref{lem:medians}. The three cases correspond to whether $x$ lies in $\{p,q,r\}$, is a distinct cusp, or is material. In every case, we compute the winding sign using \Cref{lem:prewind_identity}. 

\begin{claim}\label{claim:caseA}
    The cycle identity \eqref{eq:cycle_identity} holds when $x\in \{p,q,r\}$.
\end{claim}
\begin{proof}
    By symmetry, we may take $x = q$. Let $\alpha$ be a geodesic from $p$ to $q$, and $\beta$ from $q$ to $r$. \Cref{lem:follow_geodesics} implies that concatenating $\alpha$ and $\beta$ gives a geodesic from $p$ to $r$. Let $\gamma$ denote the reverse of their concatenation. This is a geodesic from $r$ to $p$. Then
     \[
     \Theta(\gamma) = -\Theta(\alpha)-\Theta(\beta)+ K
     \]
     where $K$ is either $\omega_q(\gamma)$ or $\sigma_q(\alpha)$. These lie in $\{-1,1\}$, so in both cases \Cref{eq:cycle_identity} holds. \qedhere

         
    \begin{figure}[h]
        \centering
        \includegraphics[width=0.6\linewidth]{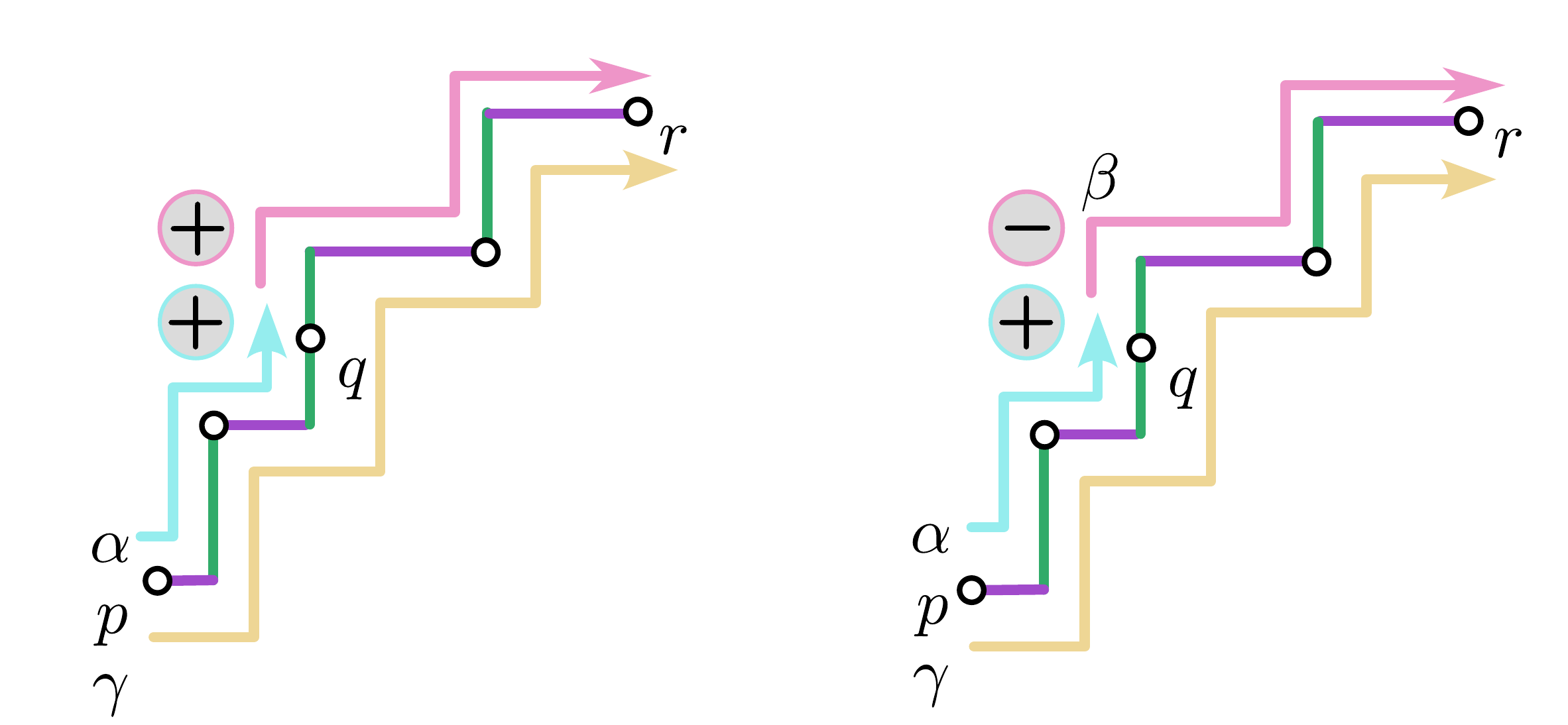}
        \caption{The geodesics for \Cref{claim:caseA}.}
        \label{fig:caseA}
    \end{figure}
\end{proof}

For the next two cases, let $x\notin\{p,q,r\}$.  Let $\alpha_x$, $\beta_x$ and $\gamma_x$ be geodesic paths from $p, q, r$ to $x$, respectively. Let $\alpha$, $\beta$ and $\gamma$ be geodesics from $p$ to $q$, $q$ to $r$ and $r$ to $p$ obtained by concatenating each of the three geodesic paths with the appropriate reverse of another. These are geodesics by \Cref{lem:follow_geodesics}.

\begin{claim}\label{claim:caseB}
    The cycle identity \eqref{eq:cycle_identity}  holds when $x$ is a cusp and $x\notin \{p,q,r\}$.
\end{claim}
\begin{proof}
    Consider the unstable signs of the truncated geodesics as they approach $x$. Then either all geodesics approach with the same sign, or two agree and the other disagrees. Using the cyclic and inversion symmetries, we suppose that $\beta_x$ and $\gamma_x$ are positive. Then we have two cases according to the sign of $\alpha_x$. Consulting \Cref{fig:caseB}, we have:
     \begin{itemize}
         \item Two positive, one negative: $\alpha_x$ approaches negatively. Then:
         \begin{align*}
             \gw(p,q) &= \gw(\alpha_x) - \gw(\beta_x) + 1\\
             \gw(q,r) &= \gw(\beta_x)-\gw(\gamma_x) + \omega_x(\beta)\\
             \gw(r,p) &= \gw(\gamma_x) - \gw(\alpha_x) - 1.
         \end{align*}
         So their sum is $\omega_x(\beta)$. 
         \item All positive: $\alpha_x$ approaches positively. Then:
         \begin{align*}
             \gw(p,q) &= \gw(\alpha_x) - \gw(\beta_x) + \omega_x(\alpha)\\
             \gw(q,r) &= \gw(\beta_x)-\gw(\gamma_x) + \omega_x(\beta)\\
             \gw(r,p) &= \gw(\gamma_x) - \gw(\alpha_x) + \omega_x(\gamma)
         \end{align*}
         Since the incoming unstable signs at $x$ are positive, the $\omega_x$-terms are just $\hat{\omega}_x$-terms. The orthants around $x$ are linearly ordered, and so at least one of the winding terms is positive and another negative. Thus the triple sum of the winding terms is in $\{-1,1\}$. \qedhere
        \end{itemize}
\end{proof}
\begin{figure}[h]
        \centering
        \includegraphics[width=0.85\linewidth]{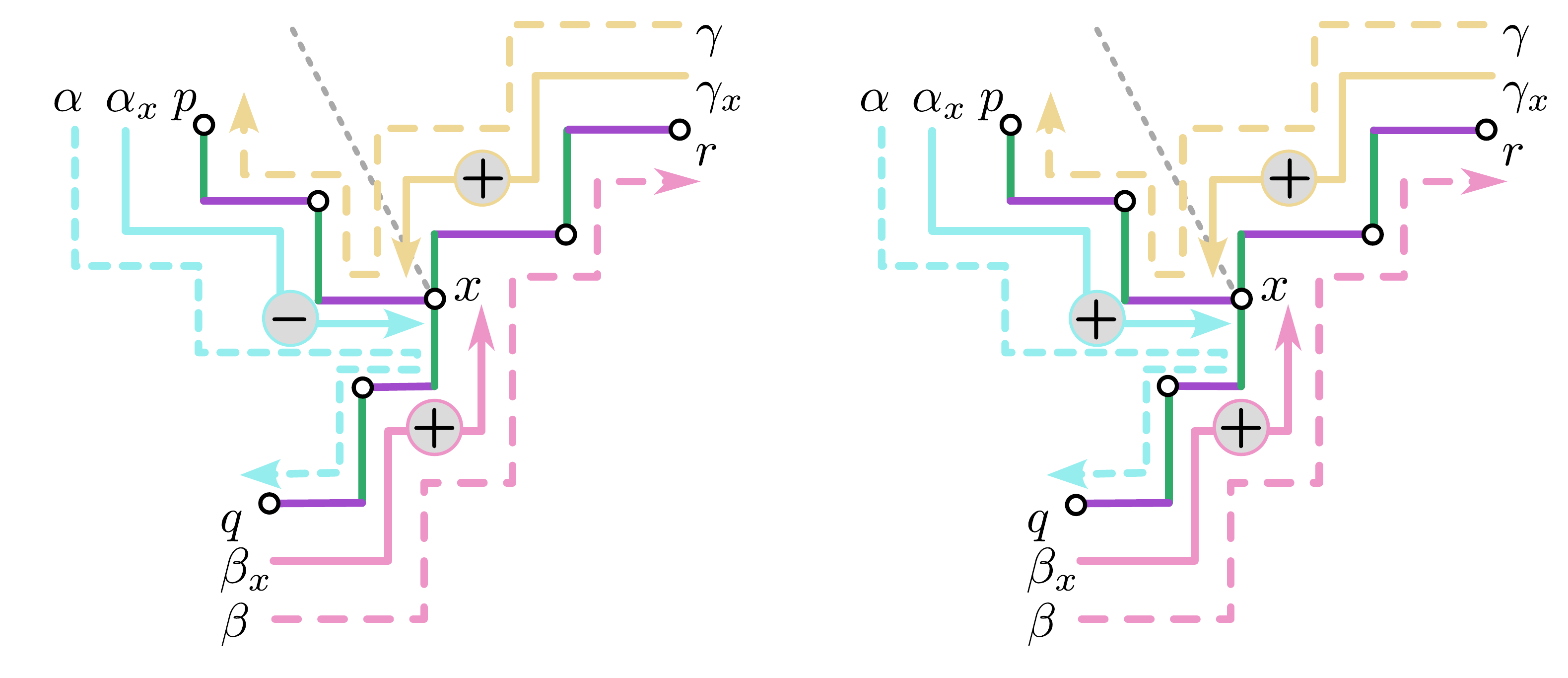}
        \caption{The geodesics for \Cref{claim:caseB}.}
        \label{fig:caseB}
    \end{figure}

The remaining case, where $x$ is material, requires additional analysis.

\begin{claim}\label{claim:turns}
        For one of $\alpha_x,\beta_x$ and $\gamma_x$, the vertex before $x$ is material. For the other two, the vertex before $x$ is a cusp.
     \end{claim}
     \begin{proof}
         There are four distinct half-leaves through $x$. If any two of the geodesics approach $x$ along the same half-leaf, then one of the concatenations must turn back on the same leaf through $x$, contradicting that it is a geodesic. So all three geodesics approach $x$ in different half-leaves through $x$.
         
         Since leaves of a loom space meet at most one cusp, the corner on at least one of the geodesics before $x$ is material. Say that it is $\beta_x$ that takes a material turn. If $\alpha_x$ also takes a material turn, then it would turn into a leaf that meets $\beta_x$ or $\gamma_x$. This would give $\beta$ or $\gamma$ a U-turn, respectively.  So $\alpha_x$ turns at a cusp before $x$. Similarly, $\gamma_x$ turns at a cusp before $x$. 
     \end{proof}
     
\begin{claim}\label{claim:caseC}
    The cycle identity \eqref{eq:cycle_identity}  holds when $x$ is material. 
\end{claim}
\begin{proof}
    By \Cref{claim:turns}, one of the three geodesics takes a material turn before $x$. By cyclic symmetry, suppose that it is taken by $\beta_x$. One of $\alpha_x$ and $\gamma_x$ approaches $x$ along an unstable leaf. Up to relabelling, suppose that it is $\gamma_x$. Now let $a,\,b,\,c\in\cusps(\cal{L})$ be the last cusps before $x$ along $\alpha,\, \beta,\,\gamma$, respectively. Since geodesics cannot change unstable direction at a material point, the unstable direction of $\beta_x$ from $b$ to $x$ is the opposite of the direction from $c$ to $x$ in $\gamma_x$. Invariance of the identity under flipping the signs of the foliations allows us to fix the direction from $b$ to $x$ as positive in the unstable direction. 

     Let $s(\alpha)$ be the unstable direction of $\alpha$ from $a$ to $b$. We decompose $\alpha$ at both $a$ and $b$ and use \Cref{lem:prewind_identity} to compute:
    \[
    \gw(\alpha) = \gw(\alpha_a)-\gw(\beta_b) +d(\alpha),
    \]
    where the defect $d(\alpha)$ takes values as in the table:

    \begin{table*}[h]
        \centering
        \begin{tabular}{|c|c|c|}
        \hline
        
        & $\alpha$ turns at $a$
        &  $\alpha$ does not turn at $a$
        \\
        \hline
        $\alpha$ turns at $b$
        & $\omega_a(\alpha)$+$\omega_b(\alpha)+s(\alpha)$
        & $\omega_b(\alpha)$ \\
        
        $\alpha$ does not turn at $b$
        & $\omega_a(\alpha)$
        & $-s(\alpha)$ \\
        \hline
        \end{tabular}
    \end{table*}
    Our configuration has $s(\alpha)=-1$. We define $d(\beta)$ and $d(\gamma)$ similarly. These can be computed with a similar table. Our configuration has $s(\beta)=+1$ and $s(\gamma)=-1$. Then
    \[
    \Theta(p,q)+\Theta(q,r)+\Theta(r,p) = d(\alpha)+d(\beta)+d(\gamma).
    \]
    Suppose first that $a,b,c\notin \{p,q,r\}$. Varying the signs of the truncated geodesics as they approach their respective endpoints gives us eight cases to compute the triple sum. See \Cref{fig:caseC}, left. We use an ordered triple to denote the signs of $\alpha_a,\,\beta_b,\,\gamma_c$ as they approach their endpoints, respectively. We compute the defect terms and summarise them in \Cref{table:sums}. Within the sums, we use that winding terms of different geodesics at $b$ and $c$ cancel. 
    %
    %
    %
   \begin{table}[h]
        \centering
        \begin{tabular}{|c|c|c|c|c|}
        \hline
        $(\alpha_a,\beta_b,\gamma_c)$
        & $d(\alpha)$
        & $d(\beta)$
        & $d(\gamma)$
        & \text{Sum} \\
        \hline
        
        $(+,+,+)$
        & $\omega_a(\alpha)$
        & $\omega_c(\beta)$
        & $\omega_c(\gamma)$
        & $\omega_a(\alpha)$ \\
        
        $(-,+,+)$
        & $1$
        & $\omega_c(\beta)$
        & $\omega_a(\gamma)+\omega_c(\gamma)-1$
        & $\omega_a(\gamma)$ \\
        
        $(+,+,-)$
        & $\omega_a(\alpha)$
        & $-1$
        & $1$
        & $\omega_a(\alpha)$ \\
        
        $(-,+,-)$
        & $1$
        & $-1$
        & $\omega_a(\gamma)$
        & $\omega_a(\gamma)$ \\
        
        $(+,-,+)$
        & $\omega_a(\alpha)+\omega_b(\alpha)-1$
        & $\omega_b(\beta)+\omega_c(\beta)+1$
        & $\omega_c(\gamma)$
        & $\omega_a(\alpha)$ \\
        
        $(-,-,+)$
        & $\omega_b(\alpha)$
        & $\omega_b(\beta)+\omega_c(\beta)+1$
        & $\omega_a(\gamma)+\omega_c(\gamma)-1$
        & $\omega_a(\gamma)$ \\
        
        $(+,-,-)$
        & $\omega_a(\alpha)+\omega_b(\alpha)-1$
        & $\omega_b(\beta)$
        & $1$
        & $\omega_a(\alpha)$ \\
        
        $(-,-,-)$
        & $\omega_b(\alpha)$
        & $\omega_b(\beta)$
        & $\omega_a(\gamma)$
        & $\omega_a(\gamma)$ \\
        
        \hline
        \end{tabular}
        \caption[The triple sum in the material cases.]{The triple sum in the material cases according to the incoming signs of the three truncated geodesics at their endpoints.
        }
        \label{table:sums}
    \end{table}

    When at least one of $a,b,c$ is equal to one of $p,q,r$, the computations are similar, but not quite identical. When $a\neq p$, the defects are the same by setting $\omega_b$ and $\omega_c$ terms to be zero if $b=q$ or $c=r$, respectively. The triple sum is then the same as in \Cref{table:sums} in this case. In the case $a=p$, the sum of defects is always $-1$. See \Cref{fig:caseC}, right. We omit the computation for the sake of exposition. 
\end{proof}

We conclude the proof of the cycle identity. 
\begin{proof}[Proof of \Cref{lem:cycle_identity}]
Combine Claims~\ref{claim:caseA}, \ref{claim:caseB}, and \ref{claim:caseC}.
\end{proof}

\begin{figure}[ht]
        \centering
        \includegraphics[width=0.7\linewidth]{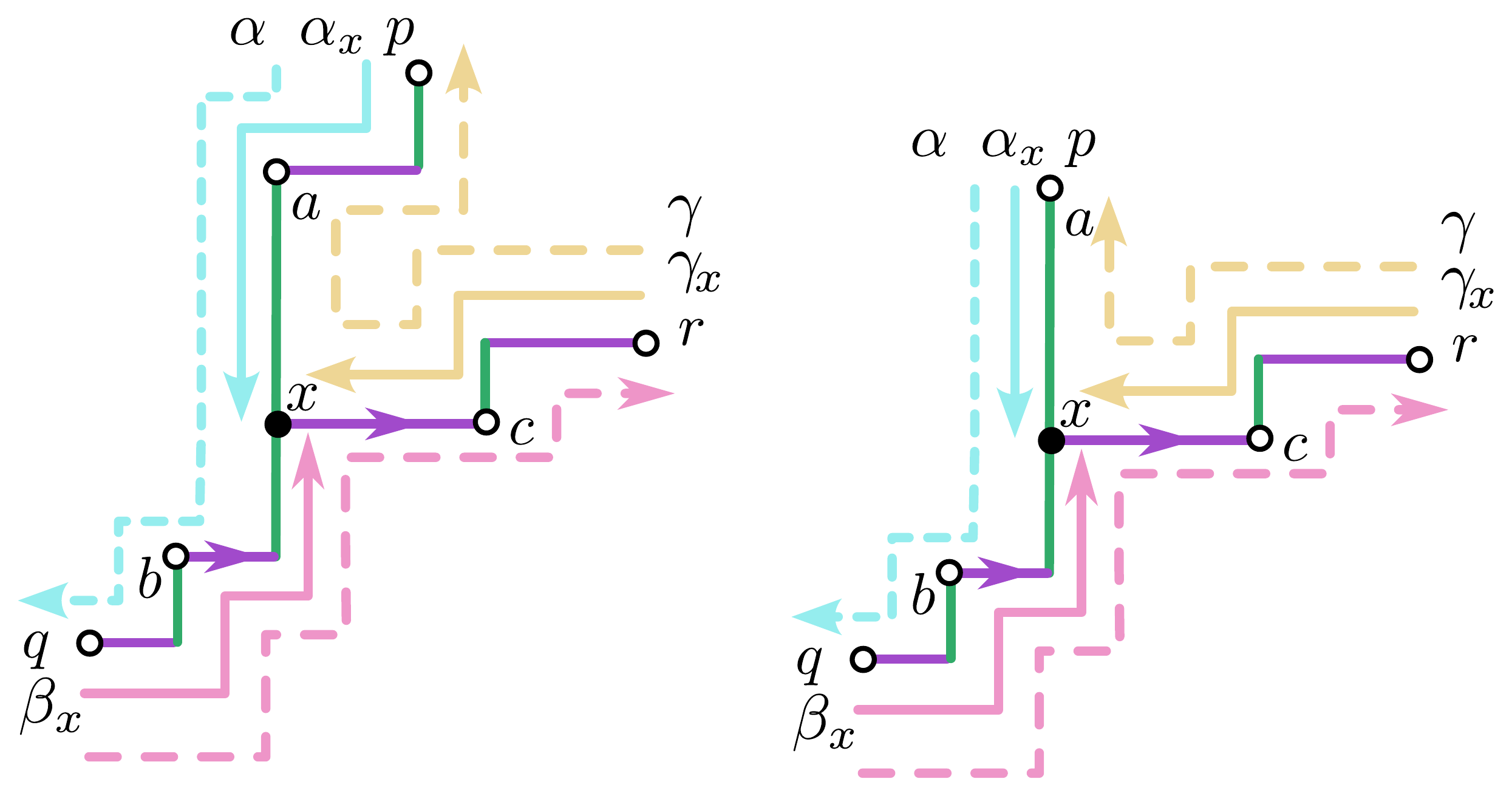}
        \caption{The geodesics for \Cref{claim:caseC}.}
        \label{fig:caseC}
\end{figure}

\begin{defn}\label{def:winding_order}
    For $p,q\in\cusps(\cal{L})$ distinct cusps, set $ p<_{\gw} q$ if $\gw(p,q)>0$.
\end{defn}

\begin{prop}\label{prop:total_order}
    The relation $<_{\gw}$ is a $\pi_1(M)$-invariant total order on $\cusps(\cal{L})$.
\end{prop}
\begin{proof}
    Everything except transitivity is given by \Cref{lem:order_basics}.
    If $p\windingorder q$ and $q\windingorder r$, then \Cref{lem:cycle_identity} gives
    \[
    \gw(p,r)  = \gw(p,q)+\gw(q,r)\pm 1. 
    \]
    Since $\gw$ takes odd-integer values, $\gw(p,r)>0$. So $p\windingorder r$.
\end{proof}

Using the invariant identification between $\cusps(\cal{L})$ and $\cusps(\lift{M})$, we have:
\begin{thm}\label{thm:totalorder}
    Suppose that $M$ admits an edge-orientable veering triangulation $\cal{V}$. Then there is a $\pi_1(M)$-invariant total order on $\cusps(\lift{M})$. \qed
\end{thm} 

We finish by relating the total order to the partial orders induced by veering branched surfaces. This requires some familiarity with the Agol-Gu\'eritaud construction \parencites[\S 6]{loom_new_arxiv}[\S 4]{landry-minsky-taylor}{ag-talk}.

\begin{prop}\label{prop:compatible_order}
    Let $M$ admit an edge-orientable veering triangulation $\cal{V}$. Let $B$ be either of the cooriented veering branched surfaces of $\cal{V}$. Then there is a $\pi_1(M)$-invariant total order on $\cusps(\lift{M})$ that is compatible with $B$.
\end{prop}
\begin{proof}
    Assume $B=B^s$ is the stable branched surface; the unstable case is similar. Since $\cal{V}$ is edge-orientable, $<_{\Theta}$ is a $\pi_1(M)$-invariant total order by \Cref{prop:total_order}. 

    Fix $e$ an edge of $\cal{V}$. Then $<_{\Theta}$ orients $e$ to point to the east-most cusp in the boundary of the edge rectangle corresponding to $e$ by \Cref{lem:order_basics}. In the Agol-Gu\'eritaud construction, every tetrahedron in $\cal{V}$ corresponds to a \emph{tetrahedron rectangle} in $\cal{L}$. For the four types of veering tetrahedron, one can check that the east-most edge direction imposed by its tetrahedron rectangle coorients $B^s$. Inverting $<_{\Theta}$ if necessary, we obtain the claim.
    %
\end{proof}
\begin{rem}
    The compatible total order on $\cusps(\lift{M})$ is not unique. Our construction can be modified to give several such orders. First, we can replace $\omega$ with $\hat{\omega}$. The proof that this defines a total order is nearly identical.
    
    In either of the orders above, we can negate one of the sums in the definition of $\gw$ (\Cref{def:turning_number}) and obtain a distinct order. If $\cal{V}$ has multiple vertices, we can choose a sign independently for each cusp orbit. 
    %
\end{rem}

\subsection{Veering triangulations give branchings}\label{subsec:vts_give_branchings}
We prove that veering triangulations yield branchings, establishing \Cref{thm:smoothcriterion}. To do so, we want to find branchings, not just cooriented branchings. So, we no longer assume that $\cal{V}$ is edge-orientable. Its edge-orientable double cover thus admits a cooriented branching and a symmetry that preserves the veering structure but inverts the edge orientations of the branching.


\begin{lemma}\label{lem:twisted_order_flip}
    For $g\in \pi_1(M)$, suppose that the induced action of $g$ on $\cal{L}$ reverses the orientations of the foliations of $\cal{L}$. Then
    \[\Theta(g\cdot p,g\cdot q) = -\Theta(p,q).\]
\end{lemma}
\begin{proof}
    Let $\gamma$ be a geodesic from $p$ to $q$. Since $g$ preserves the foliations, $g\cdot \gamma$ is a geodesic from $g\cdot p$ to $g\cdot q$. The reversed foliation directions imply that $\sigma$-terms in $\Theta(g\cdot p, g\cdot q)$ are negations of the $\sigma$-terms in $\Theta(p,q)$. Since $g$ preserves the orientation on $\cal{L}$, $\hat{\omega}$ is invariant under $g$. So, the direction reversal on the foliations implies that $g$ negates the $\omega$-terms.
\end{proof}

\begin{thm}\label{thm:smoothcriterion}
    Suppose that $M$ admits a veering triangulation. Let $\cal{T}$ be an essential ideal triangulation of $M$. Then $\cal{T}$ is branchable. 
\end{thm}

\begin{proof}
     Let $\cal{V}$ be a veering triangulation of $M$. Suppose that $\cal{V}$ is edge-orientable. \Cref{thm:totalorder} gives an invariant total order on $\cusps(\lift{M})$, and hence a coorientable branching on $\cal{T}$ by \Cref{prop:cusps_give_branching}. 
    
    Now suppose that $\cal{V}$ is not edge-orientable. Lift $\cal{V}$ to its edge-orientable double cover  $\overline{\cal{V}}$. Let $\overline{M}$ be the corresponding double cover of $M$. By \Cref{thm:totalorder}, $<_{\Theta}$ is a $\pi_1(\overline{M})$-invariant total order.

    For $\cal{T}$ an essential triangulation of $M$, let $\overline{\cal{T}}$ be its lift under the edge-orientable double cover of $\cal{V}$. Apply \Cref{prop:cusps_give_branching} to endow $\overline{\cal{T}}$ with the branching $\overline{B}$ that is compatible with $<_{\Theta}$. Let $h:\overline{M}\to\overline{M}$ be the deck transformation of the double cover. Then $h$ is a veering automorphism of $\overline{\cal{V}}$ that inverts the edge orientations of $\overline{\cal{V}}$. Let $\cal{L}(\overline{\cal{V}})$ be the loom space of $\overline{\cal{V}}$. Then the induced action of $h$ on $\cal{L}(\overline{\cal{V}})$ inverts the orientations on both foliations. By \Cref{lem:twisted_order_flip}, applying $h$ reverses $<_{\Theta}$. So, $h$ reverses the edge orientations on $\overline{\cal{T}}$. By \Cref{lem:smooth_flipped}, $h$ preserves the underlying smoothings on the tetrahedra of $\overline{\cal{T}}$. Thus the branching $\overline{B}$ descends to $\cal{T}$. 
\end{proof}

\begin{rem}\label{rem:reason_for_twisted}
    \Cref{lem:twisted_order_flip} is precisely why we choose to use the twisted winding sign, $\omega$, in the definition of $\gw$. Using only $\hat{\omega}$, only one of the sums in $\gw$ would invert under flipping the orientation on the foliations of $\cal{L}$. So, we would not obtain a branching that descends through the edge-orientable double cover. 
\end{rem}

\begin{rem}\label{rem:order_compute}
We can `see' the total order explicitly. Let $\cal{V}$ be a veering triangulation. Let $<_{\gw^s}$ and $<_{\gw^u}$ be the total orders using the $<_{\Theta}$ construction for the stable and unstable foliations on the loom space associated to $\cal{V}$. Frankel-Schleimer-Segerman show that $\lift{\cal{V}}$ admits a \emph{continental exhaustion}. In short, this means that starting with a single tetrahedron, one can exhaust the universal cover by continually layering tetrahedra above and below it \cite[\S 4-5]{ss-link}. A \emph{continent} $C$ is a ball obtained at some step of this process. Using the correspondence between veering triangulations and loom spaces, the structure of $C$ provides enough data to compute both $<_{\gw^u}$ and $<_{\gw^s}$ on the subset of cusps corresponding to vertices of $C$.

We implement this procedure \cite{me_data}. Example approximated total orders are shown in \Cref{fig:total_order}.
\end{rem}


\begin{figure}[h]
    \centering
    \includegraphics[width=0.6\linewidth]{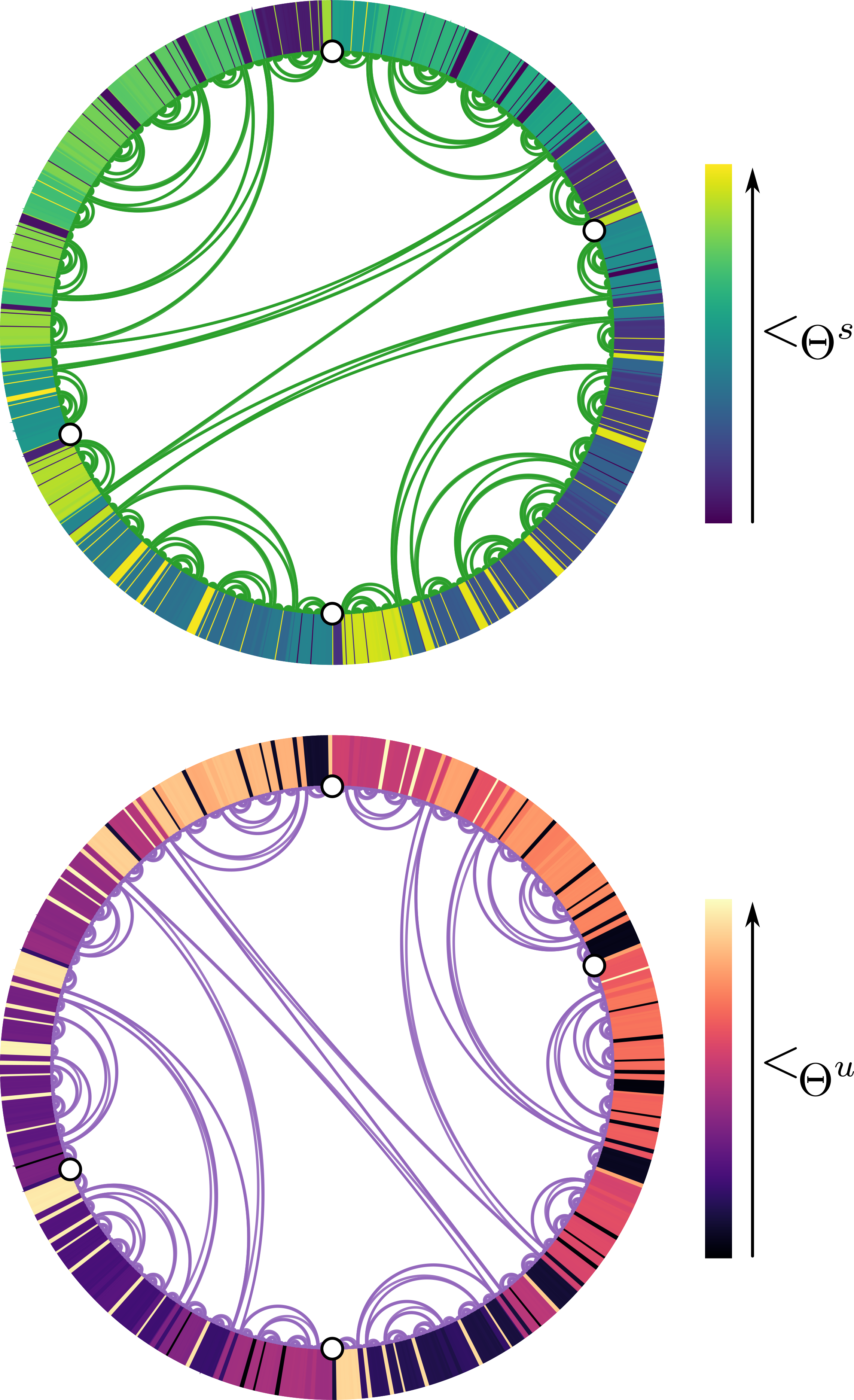}
    \caption[Total cusp orders on a loom space.]{An approximation of total cusp orders $<_{\gw^u}$ and $<_{\gw^s}$ described in \Cref{rem:order_compute}. The orders are evaluated on a large continent $C$ in the universal cover of the figure-eight knot complement. Each shaded sliver represents a cusp on the boundary circle. The colour shows its rank in the marked total order among the cusps in $C$. The four marked vertices are those of the initial tetrahedron. The interior of each disk shows the continent's approximation of the stable and unstable foliations (blown up to laminations) of the loom space. See \cite[\S 5-10]{ss-link} for more details on this framework.
    }
    \label{fig:total_order}
\end{figure}

\section{Applications}\label{sec:examples}

\subsection{Census manifolds}\label{subsec:census}
To determine whether a given triangulation is branchable, we perform a simple depth-first search. We assign branchings to tetrahedra one-by-one, propagate any implied structure, and discard a partial assignment if it yields a contradiction. We note that Dunfield found foliar edge orientations by encoding the problem as a satisfiability problem \cite[\S 7.9]{dunfield-foliar}. It may be possible to do something similar in our generalised case, but the depth-first search is sufficiently fast for our purposes. The implemented check uses features from both \emph{Regina} and \emph{Veering} \cite{regina,Veering}. 

We apply this check to Shana Li's census of orientable cusped hyperbolic three-manifolds up to eleven tetrahedra \cite{ shana_11_tets, shana_10_tets}. Since the triangulations are geometric, they are essential. However, because unbranchability depends on the triangulation (\Cref{eg:m006}), we also check a collection of triangulations obtained by random sequences of Pachner moves. For such examples, we certify essentiality using the existence of a taut structure.

 We apply \Cref{thm:smoothcriterion} to the unbranchable triangulations that we find.

\begin{thm}

\label{thm:census_no_veering}
    At least 165,065 manifolds in the census of orientable cusped hyperbolic three-manifolds up to eleven tetrahedra do not admit veering triangulations. \qed
\end{thm}

See \Cref{fig:census}. The data and code are available at \cite{me_data}.

\begin{figure}[h]
    \centering
    \includegraphics[width=0.9\linewidth]{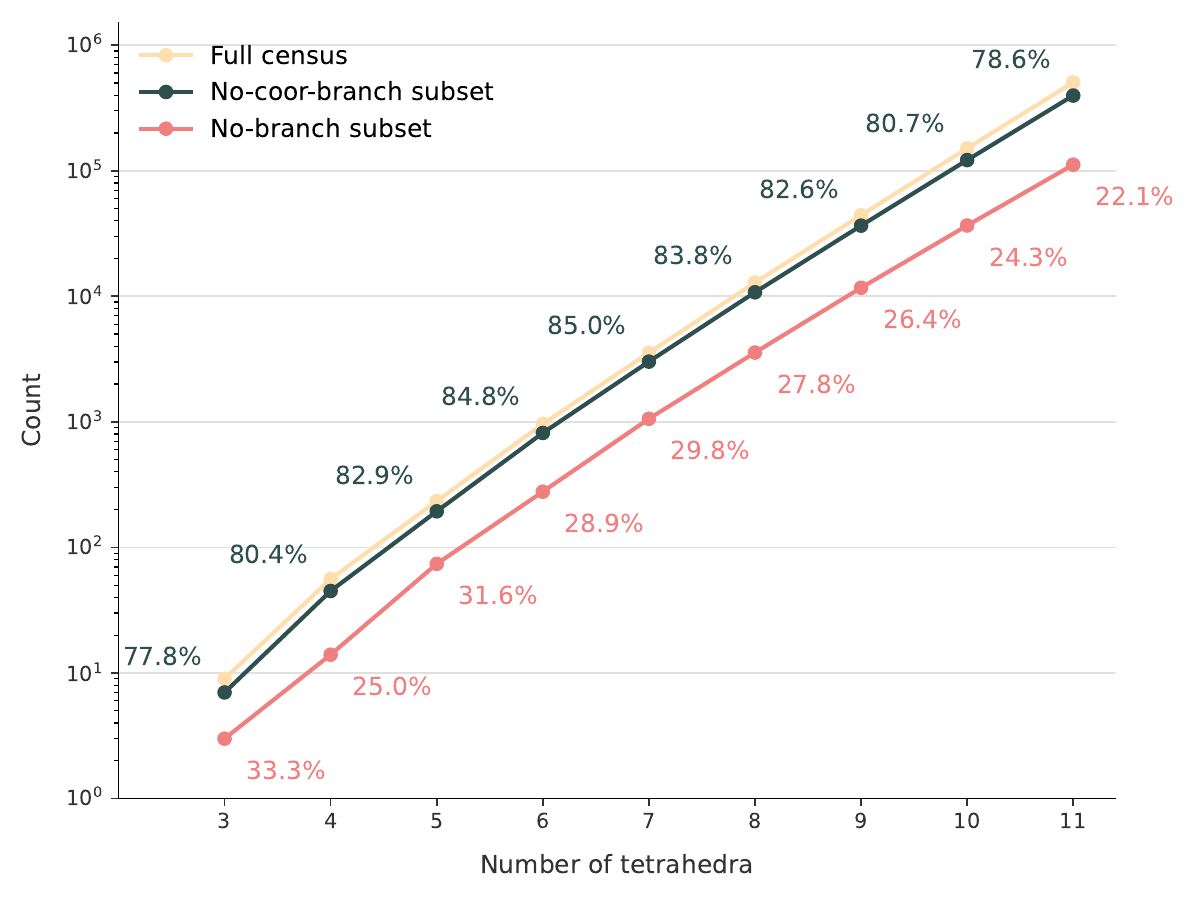}
    \caption[Census manifolds without veering triangulations.]{The distribution of three-manifolds for which we find unbranchable essential triangulations. For each tetrahedron count, we mark the percentage of manifolds for which we find either unbranchable triangulations or triangulations that are not coorientably branchable.}
    \label{fig:census}
\end{figure}

Triangulations that are not coorientably branchable are more prevalent. We find such census examples by augmenting the original depth-first search for branchings. Applying \Cref{prop:bi-order_branching} to these examples, we obtain:
\begin{thm}
\label{thm:biordercensus}
    At least 570,207 manifolds in the census of orientable cusped hyperbolic three-manifolds up to eleven tetrahedra do not have bi-orderable fundamental groups. The same manifolds also do not admit edge-orientable veering triangulations.\qed 
\end{thm}

We note that large collections of manifolds without edge-orientable veering triangulations were already known, since the exterior of an L-space knot or link cannot admit such a structure (see, for example, \cite{misha-dunfield-egs}). Still, unbranchability is easy to detect.

%
\begin{example}

Dunfield found that at least 8,153 manifolds in the cusped census up to nine tetrahedra admit persistently foliar triangulations \cite[\S 8]{dunfield-foliar}. We are unable to find a triangulation without a coorientable branching for any of these manifolds. So, we conjecture that a persistently foliar structure furnishes a total cusp order. Indeed, this appears to be analogous to the foliation-implies-left-orderable direction of the L-space conjecture.

Let $K\subset S^3$ be a knot and let $M$ be the exterior of $K$. Dunfield conjectures, based on numerical data and work of Delman-Roberts, that every non-L-space knot exterior admits a persistently foliar triangulation \cite[\S 1.7]{dunfield-foliar}. Among census manifolds that are knot exteriors, we find unbranchable triangulations only for the exteriors of L-space knots.

As a consequence, we do not shed light on the bi-orderability of knot groups, since it is known that the fundamental group of an $L$-space knot complement is not bi-orderable \cite[Theorem 1.2]{clay-rolfsen}. 

There are precisely 632 hyperbolic L-space knots whose complements admit triangulations with at most nine tetrahedra \cite{dunfield_L_space_knots} and \cite[Theorem~5.1]{baker_kegel_mccoy_search}. At least 325 of these have triangulations that are not branchable. Moreover, for all but one of these 632 knot complements, we find essential triangulations that are not coorientably branchable. The exception is \texttt{o9\_30634}. The exceptional knot itself is also the only one of the 632 that is not a positive braid closure \cite{kegel-baker}.
\end{example}

\subsection{Infinite families}\label{subec:inf_families}
We now build an infinite family to prove \Cref{thm:infiniteegs}. We leverage the fact that branchability can be obstructed locally, in the following sense. Given a triangulation $\cal{T}$, a \emph{subtriangulation} $\cal{T}'$ of $\cal{T}$ is a three-complex obtained by ungluing a subset of faces and deleting a subset of tetrahedra of $\cal{T}$. Branchability is defined in the straightforward way for subtriangulations. Moreover, a branching restricts to a branching of a subtriangulation, so we deduce:

\begin{lemma}\label{lem:subtri}
    Suppose that $\cal{T}$ has a subtriangulation that is not branchable. Then $\cal{T}$ is not branchable. \qed 
\end{lemma}
To build infinitely many examples, we build infinitely many triangulations that contain a common unbranchable subtriangulation. We use an operation that is a generalisation of Schleimer-Segerman's unpublished \emph{veering Dehn surgery} \cite{ss-surgery}. In its simplest form, the operation takes as input a veering triangulation with a M\"obius band face, pops open a face, and inserts veering tetrahedra. The construction is implemented in the Veering codebase \cite{Veering}. The operation was generalised and interpreted as an operation on \emph{almost-veering} branched surfaces by Tsang \cite[\S 3]{tsang}. We generalise it to taut triangulations and describe an associated drilling operation. The construction requires some basic familiarity with \emph{layered solid tori} \cite[\S 4]{layered_ts}. 

Let $\cal{T}$ be taut. Let $e$ be an edge of $\cal{T}$. Consider the pairs $(f,e_i)$ where $f$ is a face of $\cal{T}$ and $e_i$ is an edge of $f$ that is identified to $e$. The $\pi$-edges at $e$ partition these pairs into two fans.
\begin{defn}
Let $\cal{T}$ be a taut ideal triangulation. Let $f$ be a face of $\cal{T}$ that glues to itself to form a M\"obius band. Call the two identified edges of $f$ the \emph{band edges}. Call the remaining edge the \emph{base}. Suppose that the face $f$ further satisfies the following two conditions.
    \begin{itemize}
        \item Sharpness: Let $e$ be the edge of $\cal{T}$ where the band edges of $f$ meet. We impose that the band edges of $f$ are in the same fan of $e$.  
        
        \item Distinct $\pi$-edges: For each tetrahedron $t$ on either side of $f$, the face gluing from $t$ to $f$ identifies a $\pi$-edge of $t$ to an edge of $f$. We impose that each band edge of $f$ is the image of one of these $\pi$-edges.
    \end{itemize}
    We say that $f$ is \emph{surgery admissible}. 
\end{defn}
    For example, a M\"obius band face of a (transverse) veering triangulation is always surgery admissible.

\begin{con}[Taut M\"obius surgery]\label{con:mobius_surgery}
     
    Let $\cal{T}$ be a taut ideal triangulation and $f$ a surgery admissible face of $\cal{T}$. Blow up the face $f$ into a triangular pillow: this gives a solid torus with ideal vertices. The boundary is a punctured torus tiled by two copies of the original triangle. Now remove the solid torus, and attach a new solid torus formed by a single tetrahedron $t$ with a self-gluing. See \Cref{fig:mobius_surgery}. Consider the angle structure after ungluing. The band edge blows up into two edges. Since $f$ is surgery admissible, one of the two blown-up band edges has no $\pi$-edges, while the other has two. So, we endow $t$ with a taut structure that gives the former edge two $\pi$-edges. 
    
    Attaching $t$ gives an oriented triangulation $\cal{T}_1$ with a taut angle structure. This implies that  $\cal{T}_1$ is a taut ideal triangulation of a compact three-manifold with torus boundary. 
    %

    \begin{figure}[h]
        \centering
        \includegraphics[width=0.7\linewidth]{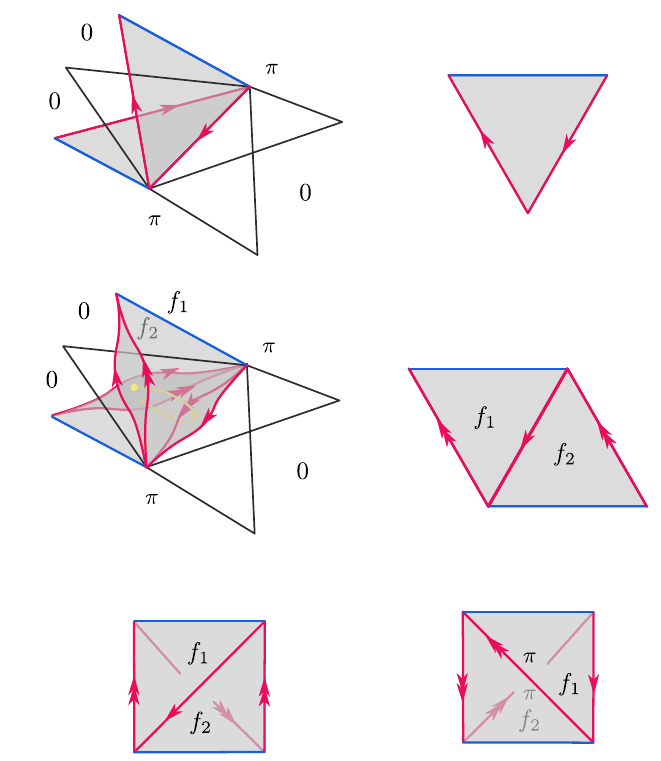}
        \caption[The taut M\"obius surgery operation.]{The taut M\"obius surgery operation. Above: an admissible face for the surgery. Middle: The picture on the universal cover of the solid torus obtained after blowing up the triangle. The yellow arc is a core longitude. Below left: the single tetrahedron that we attach to the torus boundary. Below right: the tetrahedron rotated into position as a transverse taut tetrahedron with coorientation pointing out of the page.}
        \label{fig:mobius_surgery}
    \end{figure}
    
    Topologically, we have drilled a neighbourhood of the core of the M\"obius band and attached a solid torus: this is a Dehn surgery operation on the core curve of $f$. 
    
    The attached solid torus has a M\"obius band face at its core. This face satisfies the surgery criteria above. Thus, we can iterate the surgery operation to obtain a taut triangulation $\cal{T}_n$. After the first surgery, the newly added tetrahedron opens the M\"obius face and modifies the filling slope by adding twice the core of the band. See \Cref{fig:iter_surg}. So, the manifolds $M_n$ are obtained by Dehn surgery on the original manifold, with slopes increasing in magnitude as $n\to\infty$. 

    \begin{figure}[h]
        \centering
        \includegraphics[width=0.7\linewidth]{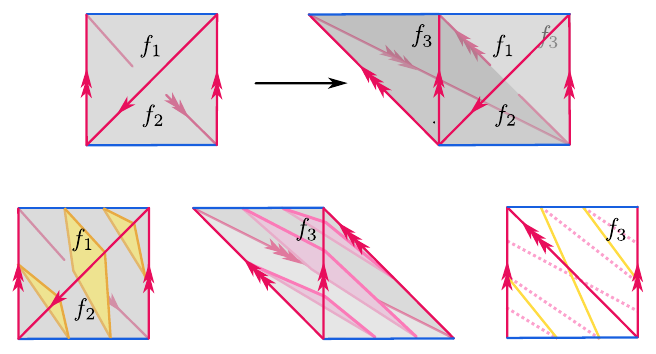}
        \caption[Iterated taut M\"obius surgery.]{One step into the iterated taut M\"obius surgery operation. Top: the newly layered tetrahedron. Below: the meridians in the first and second iterations as normal disks. Bottom right: the two filling slopes compared on the torus boundary obtained by opening up the first surgery's M\"obius band core.}
        \label{fig:iter_surg}
    \end{figure}
    
    It is also useful to build a taut triangulation of the surgery parent of the family. A \emph{drilled solid torus} is given by drilling a core longitude from a solid torus. That is, it is a copy of $\mathbb{T}^2\times I$ with one boundary component ideal. There is a two-tetrahedron ideal triangulation of a drilled solid torus with two boundary faces. After the pillow blow-up step of the surgery construction, we attach this solid torus to the two-triangle boundary to obtain a taut triangulation $\cal{T}^{\circ}$ of a manifold $M^{\circ}$ given by drilling the core of the face.
\end{con}

\begin{thm}\label{thm:infiniteegs}
     There are infinitely many cusped hyperbolic three-manifolds that do not admit veering triangulations. 
\end{thm}

\begin{proof}
    Let $M$ be the census manifold \texttt{m338}. It has a taut triangulation $\cal{T}$ and a M\"obius band triangle $f$ such that:
    \begin{itemize}
        \item The face $f$ is surgery admissible with respect to $\cal{T}$.
        
        \item The subtriangulation $\cal{T}'$ of $\cal{T}$ obtained by ungluing $f$ is not branchable.
    \end{itemize}
    The taut triangulation of $M$ is listed in \Cref{table:surgery}. The subtriangulation is given by the Regina isoSig \texttt{fLfIkadedenxqnw}.
    
    Set $\cal{T}_0:=\cal{T}$. Apply the repeated surgery operation of \Cref{con:mobius_surgery} on $f$ to obtain a sequence of taut triangulations $(\cal{T}_n)$ of manifolds $(M_n)$. 
    
    The triangulation $\cal{T}_n$ contains $\cal{T}'$ as a subtriangulation. By \Cref{lem:subtri}, $\cal{T}_n$ is not branchable. Since taut triangulations are essential, $M_n$ admits no veering triangulations by \Cref{thm:smoothcriterion}. 

    It remains to show that the manifolds $M_n$ are hyperbolic and distinct. We build the surgery parent with the drilling variant of \Cref{con:mobius_surgery} and find that it is the hyperbolic census manifold \texttt{v3140}. Since the magnitudes of the filling slopes that produce $M_n$ increase as $n\to\infty$, all but finitely many $M_n$ are hyperbolic, and infinitely many of them are pairwise distinct \cite[Theorem 5.8.2, \S 5.11]{thurston-notes}.
\end{proof}
 We record in \Cref{table:surgery} the first few triangulations in the sequence from the proof of \Cref{thm:infiniteegs}. 

\begin{table}[h]
\begin{tabular}{|l| l l l|}
\hline
$n$ & \textbf{Tetrahedra} & \textbf{SnapPy name} & \textbf{taut isoSig} \\
\hline
$0$ & {5} & \texttt{m338} & \tiny\texttt{fLLQcacdedenbxqnw\_12200} \\
\hline
$1$ & {6} & \texttt{s909} & \tiny\texttt{gLvQQadfedefjqxawjr\_100122} \\
\hline
$2$ & {7} & \texttt{v2790} & \tiny\texttt{hLvPQkccgfeefghrqhawkk\_1222001} \\
\hline
$3$ & {8} & \texttt{t07102} & \tiny\texttt{iLAvPQcbcbgfhfghxxnaxqrnw\_01100221} \\
\hline
$4$ & {9} & \texttt{o9\_16725} & \tiny\texttt{jLAwvQQbcbdhigghihhwhhxarsj\_211211002} \\
\hline
$5$ & {10} & \texttt{o10\_037294} & \tiny\texttt{kLAMvLQkbcbdeihjhijxxnxxaxqrnw\_0110100221} \\
\hline
$\infty$ & {7} & \texttt{v3140} & \tiny\texttt{hLvAQkadfdgggfjqxsjnbr\_2110202} \\
\hline
\end{tabular}

\caption[A sequence of manifolds without veering triangulations.]{The first few manifolds in the infinite sequence and the surgery parent constructed in the proof of \Cref{thm:infiniteegs}. The triangulations in the sequence are not branchable. The index $n$ indicates the iteration of the surgery operation, with $n=\infty$ denoting the surgery parent. The taut structure is specified by the taut isoSig; see \cite{Veering} for details on its meaning.}
\label{table:surgery}
\end{table}

\begin{rem}
    An unbranchable subtriangulation may appear to be an unassuming combinatorial coincidence, but it is a topological obstruction to the presence of veering triangulations. For example, it cannot be an ideal three-ball, as follows. Suppose that $\cal{T}$ is a triangulation of a three-ball with ideal vertices. Suppose that no edge of $\cal{T}$ has the same vertex at its endpoints. Then we can choose any total order on the vertices and apply \Cref{con:edge_o_to_branch} to obtain a branching on $\cal{T}$. So, the obstruction is topological, and in particular, it can disappear under finite covers.
    
    We use Regina's tricensus feature to determine that the smallest unbranchable `gadget' that can appear as a proper subtriangulation of an essential oriented ideal triangulation has at least four tetrahedra. The smallest gadgets we find among our examples have five tetrahedra. 
\end{rem}

\printbibliography

\end{document}